\documentclass[11pt, reqno]{amsproc}
\usepackage{amsmath, amsthm, amscd, amsfonts, amssymb, graphicx, color}
\usepackage{enumerate}
\usepackage{mathrsfs}
\usepackage{multirow}
\usepackage{lscape}
\usepackage{cite}
\usepackage[bookmarksnumbered, colorlinks, plainpages]{hyperref}
\hypersetup{colorlinks=true,linkcolor=red, anchorcolor=green, citecolor=cyan, urlcolor=red, filecolor=magenta, pdftoolbar=true}

\newtheorem{theorem}{Theorem}[section]
\newtheorem{lemma}[theorem]{Lemma}
\newtheorem{proposition}[theorem]{Proposition}
\newtheorem{corollary}[theorem]{Corollary}
\theoremstyle{definition}
\newtheorem{definition}[theorem]{Definition}

\theoremstyle{remark}
\newtheorem{remark}[theorem]{Remark}
\newtheorem{note}[theorem]{Note}

\numberwithin{equation}{section}

\begin{document}

\setcounter{page}{1}

\title[Dual of the Hatdy space associated to the Dunkl-Schr\"odinger operator]
 {The dual of the Hardy space associated to the Dunkl-Schr\"odinger operator with reverse H\"older class potential}

 \author[Athulya P, S.K. Verma]{
  Athulya P \and Sandeep Kumar Verma \\
\small{Department of Mathematics, SRM University AP, Amaravati 522240, Andhra Pradesh, India}\\ 
    Author's Email: \texttt{athulya.panoli97@gmail.com}, \texttt{ sandeep16.iitism@gmail.com}
}



\begin{abstract}  
Let $\mathcal{L}_k = -\Delta_k + V$ be a Schr\"odinger operator associated with the Dunkl Laplacian $\Delta_k$, where $V$ is the non-negative potential function belonging to the reverse H\"older class  $RH_k^q(\mathbb{R}^n)$ with $q> \max\{1, \frac{n+2\gamma}{2}\}$. Here, $2\gamma$ denotes the degree of homogeneity of the weight function $w_k$, which is determined by the normalized root system and the non-negative multiplicity function $k$. In this paper, we investigate the dual space of the Hardy space $H_{\Tilde{\mathcal{L}}_k}^1$ associated with the Dunkl-Schr\"odinger operator. The dual space $BMO(\mathcal{L}_k)$ is a subspace of the $BMO_k$ space, which is the Dunkl analogue of the classical $BMO(\mathcal{L})$ space. We provide a characterization for the $BMO(\mathcal{L}_k)$ space. The duality result is obtained via the atomic decomposition of $H_{\Tilde{\mathcal{L}}_k}^1$, where the cancellation condition of atoms depends on the critical radius function associated with the potential $V$. Finally, we establish the boundedness of the uncentered maximal function on the space $BMO(\mathcal{L}_k)$.
\end{abstract}
\maketitle

\section{Introduction}\label{S:1}
The theory of Hardy space traces its origin to the early twentieth century, around 1910, arising naturally in the study of Fourier series and complex analysis in one variable. From that period onward, Hardy spaces have become a fundamental tool in modern harmonic analysis, providing deep insights into core topics such as maximal functions, singular integrals, and the structure of $L^p$ spaces \cite{Stein-book-1, Stein-1960}. Hardy spaces can also be seen as the natural substitute for $L^p$ spaces when $0 < p < 1$. In this range, many fundamental operators in harmonic analysis cannot be handled effectively using the standard $L^p$ norm. This difficulty arises because functions in $L^p$ for $0 < p < 1$ do not necessarily define distributions, and the space lacks a nontrivial dual. In this setting, Hardy spaces $H^p$ provide a more suitable framework for analysis. For instance, certain Sobolev-type embedding results that fail in the $L^p$ setting for $0 < p < 1$ can be recovered by replacing $L^p$ with $H^p$ \cite{Stein-book-2}.

Functions in Hardy spaces are characterized by an atomic decomposition, where the atoms satisfy the delicate balance between integrability and cancellation properties, making them particularly well-suited for studying the boundedness of singular integral operators. Moreover, the atomic decomposition paves the way for the formation of the dual space. The celebrated result of Fefferman and Stein \cite{Fefferman-1972} identifies $BMO$ as the dual of $H^1$. The space $BMO$ is the collection of functions of bounded mean oscillation, introduced by John and Nirenberg \cite{John-Nirenberg-1961}, which has come to play a central role in harmonic analysis and partial differential equations. Apart from their favorable properties, Hardy spaces exhibit certain drawbacks. In particular, $H^p$ does not contain the Schwartz space $\mathcal{S}(\mathbb{R}^n)$, it is not well defined on manifolds, and does not admit bounded pseudo-differential operators. These difficulties are closely related to the strong cancellation property of functions in $H^p$, for instance, the fact that $\int f = 0$ for $f \in H^p$. As a consequence, multiplication by smooth functions, which may be regarded as pseudo-differential operators, does not behave well on $H^p$. These limitations create technical difficulties in various analytical settings. To address these issues, Goldberg introduced local versions of Hardy and $BMO$ spaces, denoted by $h^p$ and $bmo$, respectively \cite{Goldberg}. These spaces retain the essential features of their global counterparts while incorporating localization, thereby allowing smoother function classes such as $\mathcal{S}(\mathbb{R}^n)$ to be included. Goldberg also established the duality between $h^1$ and $bmo$, thereby extending the classical Fefferman-Stein theory to a local framework. Gradually, the Hardy space theory has been generalized far beyond the Euclidean setting. The classical Laplacian on $\mathbb{R}^n$ has been replaced by more general operators, and Euclidean spaces have been extended to manifolds and, more broadly, to metric measure spaces equipped with a doubling measure \cite{Badr, Lin, Christ}. \\

In recent decades, increasing attention has been devoted to Hardy spaces associated with differential operators, particularly the Schr\"odinger operator $\mathcal{L} = -\Delta + V,$ where $V \geq 0$ and $V \not\equiv 0$ \cite{Dziubanski-1999, Dziubanski-2005, Jacek-2004}. The introduction of the potential term significantly alters the analytical behavior of the operator,  requiring refined tools. This framework has been further generalized to families of Laplace-type operators, see for example \cite{Lin_C}. One notable example in $\mathbb{R}^n$ is the Dunkl Laplacian $\Delta_k$, introduced by Dunkl, which arises from a system of differential-difference operators associated with a finite reflection group \cite{D2}.

Dunkl operators have attracted considerable interest due to their rich algebraic structure and wide range of applications in mathematics and physics. Dunkl theory can be viewed as a far-reaching generalization of classical Fourier analysis. It includes special functions associated with root systems, spherical functions on Riemannian symmetric spaces \cite{Heckman}. For further developments, we refer the reader to the works on rational Dunkl theory \cite{Dunkl-Xu, Rosler-03}, trigonometric Dunkl theory \cite{Opdam-2000}, affine Hecke algebras and $q$-Dunkl theory \cite{Cherednik, Macdonald}, probabilistic aspects of Dunkl theory \cite{Graczyk}, and integrable systems related to Dunkl operators \cite{Etingof}. 

\par The study of Schr\"odinger operator associated with the Dunkl Laplacian $$\mathcal{L}_k=-\Delta_k + V, \quad \text{for $V\geq 0$ and $V\neq 0$}$$  was initiated by Amri and Hammi in 2018, where they developed the foundational theory \cite{Amri-2018} and later investigated the corresponding Riesz transforms and the boundedness of the Riesz transforms \cite{Amri-2021}. Subsequently, Hejna established an atomic decomposition for the Hardy space associated with the Dunkl-Schr\"odinger operator $\mathcal{L}_k$ with the reverse H\"older class $RH_k^q$ potential function $V$ \cite{Hejna-2021}. 

 In this article, we focus on identifying the dual space of the Hardy space associated with $\mathcal{L}_k$. In the classical setting, Dziuba\'nski \textit{et al.} established that the dual of the Hardy space associated with the Schr\"odinger differential operator is $BMO(\mathcal{L})$, which is a subspace of $BMO$ \cite{Dziubanski-2005}. Motivated by this result, we introduce a new function space $BMO(\mathcal{L}_k)$, which serves as a natural analogue of the classical $BMO(\mathcal{L})$ space in the Dunkl-Schr\"odinger framework. We provide a characterization of $BMO(\mathcal{L}_k)$ and show, in particular, that the local $bmo$ space can be realized as a special case of $BMO(\mathcal{L}_k)$.

In \cite{Dziubanski-2005}, the duality result is based on the atomic decomposition of the Hardy space, where the cancellation property of the atoms depends on the critical radius function. In contrast, the currently available decomposition of the Hardy space associated with $\mathcal{L}_k$ is formulated in terms of atoms whose cancellation condition depends on the size of the cubes \cite{Hejna-2021}. To establish the desired duality in our setting, we impose suitable assumptions on the potential function to obtain the corresponding atomic decomposition. Furthermore, we study the boundedness of the uncentered maximal function associated with the Dunkl weight function on the space $BMO(\mathcal{L}_k)$. In the Euclidean setting, the boundedness of the maximal operator is typically established using the Calder\'on-Zygmund  decomposition \cite{Bennett}. Although Calder\'on-Zygmund decompositions are available in spaces of homogeneous type and are useful in various contexts, the structural differences between the Euclidean framework and the present setting make their direct application less convenient. Therefore, we adapt the proof technique developed in the context of the Heisenberg group \cite{Lin}.

The paper is organized as follows. In Section \ref{S:2}, we recall the necessary preliminaries from Dunkl theory. Section \ref{S:3} is devoted to the definition and characterization of the space $BMO(\mathcal{L}_k)$. In Section \ref{S:4}, we establish the duality between the Hardy space associated with the Dunkl-Schr\"odinger operator and $BMO(\mathcal{L}_k)$. We study the behavior of the maximal function over the $BMO(\mathcal{L}_k)$ space in Section \ref{S:5}.

\section{Preliminaries}\label{S:2}
The Dunkl theory started with the seminal work of Dunkl and developed extensively afterward (see references \cite{D2, Rosler-03}). In this section, we present some preliminary materials on the Dunkl analysis and the Dunkl-Schr\"odinger operator. 
\subsection{Dunkl theory}
We consider the Euclidean space $\mathbb{R}^n$ associated with the inner product $\langle x, y \rangle = \sum_{j=1}^n x_jy_j$ and the norm $|x|^2= \langle x,x \rangle$, where $x=(x_1,x_2,\cdots, x_n )$ and $y=(y_1,y_2,\cdots, y_n)$. For a nonzero vector $\zeta \in \mathbb{R}^n$, the reflection map $\sigma_\zeta$ is given by 
\begin{align*}
    \sigma_\zeta(x)=x-2\frac{\langle x, \zeta \rangle}{|\zeta|^2}\zeta.
\end{align*} 
 A finite set of non-zero vectors $\mathcal{R}$ in $\mathbb{R}^n$ is called the root system, if for any $\zeta \in \mathcal{R}$, $\sigma_\zeta(\mathcal{R})=\mathcal{R}$, and $\mathcal{R}\cap \mathbb{R}\zeta=\{ \pm\zeta\}$. $\mathcal{R}$ can be represented as a disjoint union of sets $\mathcal{R}=\mathcal{R}_+ \cup \mathcal{R}_-$, where $\mathcal{R}_+=\{ \zeta \in \mathcal{R}: \langle \zeta, b \rangle >0\}$ for a fixed $b\in \mathbb{R}^n \setminus \bigcup\limits_{\zeta\in \mathcal{R}} \{\zeta\}^{\perp} $. This decomposition is not unique; it may vary depending on the vector $b\in \mathbb{R}^n$. We normalize the root system by $|\zeta|^2=2$ for all $\zeta \in \mathcal{R}$. The associated group generated by the reflections $\{\sigma_\zeta; \zeta \in \mathcal{R}\}$ is called the reflection group $G$. The group $G$ induces a distance function, known as the orbital distance, defined by $ d_\mathcal{O}(x,y)= \min \limits_{\sigma_\zeta\in G}\|x-\sigma_\zeta(y)\|$, where $\sigma_\zeta$ are reflections of $G$ and $\zeta\in \mathcal{R}$. The complex valued function $k$ on $\mathcal{R}$ is called the multiplicity function if $k(\zeta)=k(g(\zeta))$, for all $g\in G$. Throughout the paper, we consider $k \geq 0$. The weight function is given by $w_k(x)= \prod_{\zeta\in \mathcal{R}_+}|\langle \zeta,x\rangle|^{2k(\zeta)}$, which is group invariant and homogeneous of degree $2\gamma$, where $\gamma= \sum_{\zeta \in \mathcal{R}_+} k(\zeta)$. The weight function induces a measure on $\mathbb{R}^n$ defined by
$\omega_k(A)=\int_A w_k(x)dx$, for any Borel measurable set $A \subset\mathbb{R}^n$. It is given in \cite{Hejna-2019} that there exists positive constants $c $ and $C$ such that 
\begin{align*} 
    c r^n \prod\limits_{\zeta\in \mathcal{R}_+}\left( |\langle x,\zeta \rangle|+r\right)^{2k(\zeta)} \leq \omega_k(B(x,r)) \leq Cr^n \prod\limits_{\zeta\in \mathcal{R}_+}\left( |\langle x,\zeta \rangle|+r\right)^{2k(\zeta)},
\end{align*} where $B(x,r)$ is the Euclidean ball with center $x$ and radius $r$.
This implies that measure $\omega_k$ satisfies the doubling condition; that is, there exists a constant $C_{\omega_k}>0$ such that 
\begin{align}\label{doubling property}
    \omega_k(B(x,2r)) \leq C_{\omega_k}\omega_k(B(x,r)) \text{ for all } x\in \mathbb{R}^n \text{ and } r>0.
\end{align} In addition, there exists a constant $C\geq 1$ such that 
\begin{align} \label{Volume ratio}
    C^{-1} \left( \frac{r}{R} \right)^{n+2\gamma} \leq \frac{\omega_k(B(x, r))}{\omega_k(B(x, R))} \leq   C \left( \frac{r}{R} \right)^{n+2\gamma},  
\end{align} for $x\in \mathbb{R}^n $ and $0< r\leq R$.\\
The associated weighted Lebesgue space \(L_k^p(\mathbb{R}^n)\), for \(1\le p<\infty\), is defined by
\[
L_k^p(\mathbb{R}^n)
=
\left\{
f:\mathbb{R}^n\to \mathbb{C}\ \text{measurable} :
\int_{\mathbb{R}^n}|f(x)|^p w_k(x)\,dx < \infty
\right\},
\]
where \(dx\) denotes the Lebesgue measure on \(\mathbb{R}^n\).
For \(p=\infty\), we define
\[
L_k^\infty(\mathbb{R}^n)
=
\left\{
f:\mathbb{R}^n\to \mathbb{C}\ \text{measurable} :
\operatorname*{ess\,sup}_{x\in \mathbb{R}^n}|f(x)|<\infty
\right\}.
\] Moreover, $ L^1_{k,\text{loc}}(\mathbb{R}^n)$ is the locally integrable function over $\mathbb{R}^n$ associated  with the weight $w_k$.

For a fixed root system $\mathcal{R}$ and multiplicity function $k$, the Dunkl operator  (or differential-difference operator) \cite{D2} is defined by 
\begin{align*}
    \mathcal{T}_\xi f(x)= \partial_\xi f(x) + \sum_{\zeta \in \mathcal{R}_+}k(\zeta) \langle \zeta, \xi \rangle \frac{f(x)-f(\sigma_\zeta(x))}{\langle \zeta, x\rangle}, \text{ for } f \in \mathcal{C}^1(\mathbb{R}^n) \text{ and } \xi \in \mathbb{R}^n. 
\end{align*} Here, $\partial_\xi$ is the directional derivative in the direction of $\xi$; correspondingly, $\mathcal{T}_\xi$ can be viewed as the perturbation of the directional derivative in the Euclidean framework. Let $\{e_1, e_2, \cdots, e_n \}$ denote the standard orthonormal basis of \(\mathbb{R}^n\). For notational convenience, we write
$\mathcal{T}_j=\mathcal{T}_{e_j}; \, j=1,2,\dots,n.$ The Dunkl Laplacian \cite{D2}  is given by 
\begin{align*}
    \Delta_k = \sum_{j=1}^n \mathcal{T}_j^2.
\end{align*}
The operator $-\Delta_k$ is densely defined, positive, and symmetric over $L_k^2(\mathbb{R}^n)$ and has a unique positive extension \cite{Amri-2018}. In addition,  $-\Delta_k$ generates a semigroup $H_t$ of linear self-adjoint contractions on $L_k^2(\mathbb{R}^n)$, and it admits the following representation 
\begin{align*}
H_tf(x)=\int_{\mathbb{R}^n}h_t(x,y)f(y)w_k(y)dy,\end{align*}
where  
\begin{align*}
 h_t(x,y) = \tau_x \left( c_k^{-1}(2t)^{-(n+2\gamma)/2} e^{-|\cdot|^2/4t} \right)(-y), 
\end{align*}
The operator $\tau_x$ denotes the generalized translation operator; for further details, readers are referred to \cite{TV-15}. The heat kernel $h_t(x,y)$ is a $\mathcal{C}^\infty$ function on the variables $x,y\in \mathbb{R}^n$ and $t>0$. For all $x,y \in \mathbb{R}^n$ and $t>0$, it satisfies the following conditions \cite{Hejna-2019}:
\begin{enumerate} [$(i)$]
    \item Symmetric property: $h_t(x,y)=h_t(y,x)$ .\\
    \item  $ \int_{\mathbb{R}^n} h_t(x,y)w_k(y)dy = 1$ \\
    \item  There exists constants $C,c>0$ such that 
    \begin{align} \label{heat kernel esti_2}
        h_t(x,y) \leq C \left( 1+ \frac{|x-y|}{\sqrt{t}} \right)^{-2} \mathcal{G}_{t/c}(x,y),
    \end{align} where 
    \begin{align*}
        \mathcal{G}_{t}(x,y) = \left( \max \{ \omega_k(B(x,\sqrt{t})), \omega_k(B(y,\sqrt{t})) \} \right)^{-1} \text{exp} \left( -\frac{d_\mathcal{O}(x,y)^2}{t} \right).
    \end{align*}
\end{enumerate}
\subsection{Dunkl-Schr\"odinger operator}Given a non-negative function $V\in L_{k,\text{loc}}^1(\mathbb{R}^n)$, we consider the Dunkl-Schr\"odinger operator $\mathcal{L}_k = -\Delta_k + V$. The operator $\mathcal{L}_k$ is densely defined on the Hilbert space $L_k^2(\mathbb{R}^n)$ with domain  
\[
\operatorname{Dom}(\mathcal{L}_k) = \left\{ f \in L_k^2(\mathbb{R}^n) : \Delta_k f \ \text{and} \ M_V f \in L_k^2(\mathbb{R}^n) \right\},
\]
where $M_V f(x) = V(x)f(x)$ isfor  the multiplication operator \cite{Amri-2018}. We also consider the associated quadratic form
\begin{align*}
    \mathcal{Q}(f, g) = \sum_{j=1}^n \langle \mathcal{T}_jf, g \rangle + \langle Vf, g\rangle
\end{align*} with the domain 
$$ \text{Dom}(\mathcal{Q})= \Bigg\{f\in L_k^2(\mathbb{R}^n):  \left(\sum_{j=1}^n|\mathcal{T}_jf|^2 \right)^{1/2}  \text{and } M_{\sqrt{V}}f \in L_k^2(\mathbb{R}^n)\Bigg\}.
$$ 
Since the space $\mathcal{C}_c^\infty(\mathbb{R}^n)$ is invariant under Dunkl operators \cite{DJ1}, we have  
$\mathcal{C}_c^\infty(\mathbb{R}^n) \subset \operatorname{Dom}(\mathcal{Q})$. Hence, the quadratic form $\mathcal{Q}$ is densely defined in $L_k^2(\mathbb{R}^n)$ \cite[Lemma 4.1]{Amri-2018}. Also note that $\mathcal{Q}$ is a closed quadratic form. Consequently, there exists a unique positive self-adjoint operator $\Tilde{\mathcal{L}}_k$ such that 
\begin{align*}
    \langle \Tilde{\mathcal{L}}_kf, f \rangle = \mathcal{Q}(f, f), \quad f \in \text{Dom}(\Tilde{\mathcal{L}}_k).
\end{align*} In addition,  $ \text{Dom}(\Tilde{\mathcal{L}}_k) = \text{Dom}(\mathcal{Q}) \text{ and }  \mathcal{Q}(f,f)= \|\Tilde{\mathcal{L}}_k^{1/2}f\|^2_{L_k^2(\mathbb{R}^n)}.$
Here $\Tilde{\mathcal{L}}_k^{1/2}$ is the unique positive self-adjoint operator on $L_k^2(\mathbb{R}^n)$ such that $(\Tilde{\mathcal{L}}_k^{1/2})^2=\Tilde{\mathcal{L}}_k$ and we have the followings \cite{Amri-2018}:\\
$(i)$ $\Tilde{\mathcal{L}}_k$ is the closure of the operator $\mathcal{L}_k$.\\
$(ii)$ $\widetilde{\mathcal{L}}_k$ generates a strongly continuous semigroup of self-adjoint contractions 
$\{e^{-t\widetilde{\mathcal{L}}_k}\}_{t>0}$ on $L_k^2(\mathbb{R}^n)$. 
Moreover, for each $t>0$, the semigroup admits the integral representation
\begin{align} \label{Integral operator}
e^{-t\widetilde{\mathcal{L}}_k} f(x)
= \mathbf{K}_t f(x)
= \int_{\mathbb{R}^n} k_t(x,y)\, f(y)\, w_k(y)dy,
\end{align}
where the kernel $k_t(x,y)$ is nonnegative and satisfies the estimate
\begin{align}\label{kernel estimate_1}
0 \leq k_t(x,y) \leq h_t(x,y), \quad \text{for all } x,y \in \mathbb{R}^n, \; t>0.
\end{align}
\subsection{Critical radius function}
We assume that \(q>\max\left\{1,\frac{n+2\gamma}{2}\right\}\) and the potential function $V\geq 0$ belongs to the reverse H\"older class $RH_k^q(\mathbb{R}^n)$ \cite{Hejna-2021}, that is, there exists a constant \(C_{RH}>0\) such that
\begin{align} \label{RVC}
\left( \frac{1}{\omega_k(B)} \int_B V(x)^q\, w_k(x)\,dx \right)^{1/q}
\leq C_{RH}\, \frac{1}{\omega_k(B)} \int_B V(x)\, w_k(x)\,dx,
\end{align}
for every ball \(B \subset \mathbb{R}^n\).  
The potential \(V\) induces a function \(\rho_k\), called the critical radius function, on \(\mathbb{R}^n\), defined by
\[
\rho_k(x)
=
\sup \left\{
r>0:
\frac{r^2}{\omega_k(B(x,r))}
\int_{B(x,r)} V(\xi)\,w_k(\xi)\,d\xi
\le 1
\right\}.
\]
The function \(\rho_k\) was introduced in \cite{Hejna-2021} as a natural analogue of the classical critical radius function studied by Shen \cite{Shen-1995} and Dziuba\'nski \cite{Dziubanski-2005}. The function \(m_k(x,V)\) denotes the reciprocal of \(\rho_k(x)\).
The fundamental properties of \(m_k(x,V)\) are established in \cite{Hejna-2021}. In this subsection, we derive some additional properties of \(m_k(x,V)\) that will be useful in the sequel.
\\ \\
For each $j\in \mathbb{Z}$, we define the subsets of $\mathbb{R}^n$ in terms of the function $m_k(x,V)$ as
\begin{align*}
    {\digamma}_j = \{x\in \mathbb{R}^n: 2^{j/2} \leq m_k(x,V) < 2^{(j+1)/2} \}. 
\end{align*}Using \cite[Lemma 7]{Hejna-2021} we know that $0< m_k(x,V)< \infty$. Then $\mathbb{R}^n=\bigcup\limits_{j\in \mathbb{Z}} \digamma_j$. The following result gives the finite overlapping property of the collection $\digamma_j$ and is analogous to \cite[Lemma 2]{Dziubanski-1999}.

\begin{lemma}\label{Lemma_1}
Let $x \in \digamma_j$, for $j \in \mathbb{Z}$. Then there exists a constant $C>0$ such that for every $R>2$, we have 
\begin{align*}
    \{d: B(x, 2^{-j/2}R)\cap \digamma_d \neq \emptyset\} \subset [j-C\log_2R, j+1+C\log_2R].
\end{align*}
\end{lemma}
\begin{proof}
    Using \cite[Lemma 8]{Hejna-2021}, there exists constants $C_0, \kappa_0>0$ such that for every $x,y\in \mathbb{R}^n$ we have 
    \begin{align}
        m_k(y,V) \leq C_0 m_k(x,V) \left( 1+ m_k(x,V)|x-y|\right)^{\kappa_0}  \label{Inq_1}
    \end{align} and 
    \begin{align}
        m_k(y,V) \geq \frac{m_k(x,V)}{C_0\left( 1+ m_k(x,V)|x-y|\right)^{\kappa_0/(\kappa_0+1)} }. \label{Inq_2}
    \end{align} Let $y \in  B(x, 2^{-j/2}R)$. Since $x\in \digamma_j$, we obtain $|x-y|m_k(x, V)\leq 2R$.  
    Consequently, \eqref{Inq_1} and \eqref{Inq_2} will become
    \begin{align*}
        m_k(y, V) \leq C_0 2^{(j+1)/2} (1+ 2R)^{\kappa_0}= 2^{\frac{j+1}{2}+\kappa_0\log_2(1+2R)+ \log_2C_0} 
    \end{align*} and 
    \begin{align*}
        m_k(y, V) \geq C_0^{-1} 2^{j/2}(1+2R)^{-\kappa_0/(\kappa_0+1)}=2^{\frac{j}{2}-\frac{\kappa_0}{\kappa_0+1}\log_2(1+2R)-\log_2C_0 },
    \end{align*}respectively. Thus, we have the desired result.
\end{proof}
The following result is one of the consequences of Lemma \ref{Lemma_1} and the Vitali covering lemma. 

\begin{proposition} \label{cover_1}
    For each $j\in \mathbb{Z}$, there exists a collection of points $\{x_{(j,d)}\}_{d=1}^\infty$ such that $x_{(j,d)}\in \digamma_j$ and $\digamma_j \subset \bigcup_d B_{(j,d)}$, where $B_{(j,d)}=B(x_{(j,d)}, 2^{2-j/2})$. Moreover, there exists a constant $C>0$,   such that for all $(j, d)$ and $R \geq 2$ we have 
    \begin{align*}
        \# \left \{ (j^\prime, d^\prime):  B(x_{(j,d)}, R2^{-j/2}) \cap B(x_{(j^\prime, d^\prime)}, R2^{-j^\prime/2}) \neq \emptyset \right \} \leq R^C,
    \end{align*} where $\#$ denotes the cardinality of the set.
\end{proposition}

The following proposition gives the partition of unity associated with the balls $B_{(j,d)}$. 
\begin{proposition}\label{partition of unity}(cf. \cite{Dziubanski-1999})
    There exists a family of $\mathcal{C}^\infty$ functions $\Psi_{(j,d)}$ such that supp$(\Psi_{(j,d)}) \subset B(x_{(j,d)}, 2^{1-\frac{j}{2}}) \subset B_{(j,d)}$, $0\leq \Psi_{(j,d)}(x) \leq 1$, $\| \nabla \Psi_{(j,d)}\|_{\infty} \leq C 2^{j/2}$ and $\sum\limits_{(j,d)}\Psi_{(j,d)}=1.$ 
\end{proposition}

Throughout the paper, we assume $V\in RH_k^q(\mathbb{R}^n)$ with $V\geq 0$ and $q> \max\{1, \frac{n+2\gamma}{2}\}$,  and for each $(j,d) \in \mathbb{Z}\times \mathbb{N}$, $\Psi_{(j,d)}$ be the partition of unity given in Proposition \ref{partition of unity}. We denote by $B(x,r)^*$ the ball centered at $x$ with radius $2r$.
We denote by $C$ and $c$ positive constants, even if their values change at each occurrence.

\section{The space $BMO(\mathcal{L}_k)$}\label{S:3}
The space $BMO$ can be viewed as a natural replacement of $L^\infty$ in harmonic analysis, particularly in endpoint estimates, since it strictly contains $L^\infty$ and enjoys better stability properties under singular integral operators. For a detailed discussion, we refer the reader to \cite{Stein-book-1}. The theory of $BMO$ has been extended in various settings, for example, see \cite{Dafni-2012, Lin}. In the Dunkl framework, the corresponding $BMO$ space is defined by
\begin{align*}
BMO_k
= \left\{ f \in L_{k,\mathrm{loc}}^{1}(\mathbb{R}^n):
\sup_{B} \frac{1}{\omega_k(B)} \int_{B} \lvert f(x) - f_B \rvert \, w_k(x)\, dx < \infty \right\},
\end{align*}
where the supremum is taken over all balls $B \subset \mathbb{R}^n$ and
\[
f_B = \frac{1}{\omega_k(B)} \int_{B} f(x)\, w_k(x)\, dx .
\]
The associated seminorm is given by
\[
\|f\|_{BMO_k}
= \sup_{B} \frac{1}{\omega_k(B)} \int_{B} \lvert f(x) - f_B \rvert \, w_k(x)\, dx .
\]
Since constant functions have zero seminorm, the quotient space of $BMO_k$ modulo constants is a Banach space. Further details on $BMO_k$ spaces can be found in \cite{Jacek-2023}. Several related function spaces of $BMO$ type, such as $CMO$, $VMO$, $BLO$ etc., have been studied in the literature \cite{Almeida, Jacek-2023, Guo, Jiu}. In this section, we introduce the space $BMO(\mathcal{L})$ in the Dunkl setting and present its characterizations together with some fundamental properties.

\begin{definition}
    A function $f \in L_{k,\mathrm{loc}}^{1}(\mathbb{R}^n)$ is said to  $BMO(\mathcal{L}_k)$ function, if there exists a constant $C_f \geq 0$ such that
\[
\frac{1}{\omega_k(B(x,s))} \int_{B(x,s)} \left|f(y)-f_{B(x,s)}\right| w_k(y)\,dy \leq C_f,
\]
and
\[
\frac{1}{\omega_k(B(x,r))} \int_{B(x,r)} |f(y)|\, w_k(y)\,dy \leq C_f,
\]
for every ball $B(x,s)$ and $B(x,r)$ satisfying
\[
s<\rho_k(x)\le r,
\qquad x\in \mathbb{R}^n .
\] The smallest $C_f$ satisfying the above conditions we denote as $\|f\|_{BMO(\mathcal{L}_k)}$. Let $f \in BMO(\mathcal{L}_k)$. Then the following inequality holds:
\begin{equation}\label{bmo-BMO}
    \|f\|_{BMO_k} \leq 2\,\|f\|_{BMO(\mathcal{L}_k)}.
\end{equation}
As an immediate consequence, we obtain the continuous embedding $BMO(\mathcal{L}_k) \hookrightarrow BMO_k,$
that is, $BMO(\mathcal{L}_k) \subset BMO_k$. The inclusion is strict. For $k=0$, we have the classical example $f(x)=\log^+(|x|) \in BMO\setminus BMO(\mathcal{L})$ \cite{Bennett}. 
\end{definition}
\begin{proposition}
    The space $BMO(\mathcal{L}_k)$ is a Banach space endowed with the norm
    $\|\cdot\|_{BMO(\mathcal{L}_k)}$.
\end{proposition}
\begin{proof}
    The proof is standard and straightforward. 
\end{proof}
In the following, we discuss the properties of $BMO(\mathcal{L}_k)$ functions. To this end, we require the John–Nirenberg inequality for $BMO_k$ functions. We note that the Dunkl theory induces a metric measure space whose associated measure $\omega_k$ satisfies the doubling condition \eqref{doubling property}. Consequently, many classical results from the theory of  $BMO$ spaces on spaces of homogeneous type remain valid in this setting.

\begin{theorem}\label{John-Nirenberg} 
(John-Nirenberg inequality for $BMO_k$) Let $f\in BMO_k$ and $\lambda>0$. Then there exist positive constants $c$ and $C$ such that 
\begin{align} \label{John-Nirenberg inequality}
    \omega_k\left(\{x\in B: |f(x) -f_B| > \lambda \}\right) \leq C \omega_k(B) e^{-\frac{c\lambda}{\|f\|_{BMO_k}}},
\end{align}
for all balls $B \subset \mathbb{R}^n$.     
\end{theorem}
\begin{proof}
    The result is a particular case of \cite[Theorem 5.2]{Alto-2011}.
\end{proof}
The following are the consequences of John-Nirenberg type inequality. 
\begin{corollary} \label{Corollary J-N}
Let $f \in BMO(\mathcal{L}_k)$. Then there exist positive constants $C$ and $C^*$ such that, for every $p \in [1,\infty)$, we have
\begin{align} \label{Corollary 5.4_1}
\left(\frac{1}{\omega_k(B)} \int_B |f(x)-f_B|^p w_k(x)\,dx \right)^{\frac{1}{p}}
& \leq C \|f\|_{BMO(\mathcal{L}_k)}, \quad \text{for all balls } B,
\end{align}
and
\begin{align} \label{corollary 5.4_2}
\left(\frac{1}{\omega_k(B)} \int_B |f(x)|^p w_k(x)\,dx \right)^{\frac{1}{p}}
& \leq C^* \|f\|_{BMO(\mathcal{L}_k)}, \quad \text{for all balls } B = B(x_0, r),
\end{align}
where $x_0 \in \mathbb{R}^n$ and $r \geq \rho_k(x_0)$. Moreover, the constant $C$ depends only on the doubling constant $C_{w_k}$, $p$, and $\rho_k$.
\end{corollary}
\begin{proof}
    From the layer cake representation of a non-negative measurable function, along with the John-Nirenberg type inequality \eqref{John-Nirenberg inequality}, we have 
    \begin{align*}
        \int_B |f(x)-f_B|^pw_k(x)dx & = p \int_0^\infty t^{p-1} w_k\left( \{x\in B: |f(x)-f_B| > t \} \right)dt \\
        & \leq C \omega_k(B) p \int_0^\infty t^{p-1}e^{-\frac{ct}{\|f\|_{BMO_k}}} dt \\
        & = \frac{C \|f\|_{BMO_k}^p \omega_k(B)\Gamma(p+1)}{c^p}, 
    \end{align*} where $\Gamma(\cdot)$ denotes the gamma function. This establishes \eqref{Corollary 5.4_1} by invoking \eqref{bmo-BMO}
    \begin{align*}
         \left(\frac{1}{\omega_k(B)} \int_B |f(x)-f_B|^pw_k(x)dx \right)^{\frac{1}{p}} & \leq C\|f\|_{BMO_k}, \,\text{  for all balls } B.   
    \end{align*}The inequality \eqref{corollary 5.4_2} can be derived from \eqref{Corollary 5.4_1} by applying the Minkowski's inequality
    \begin{align*}
          \left(\frac{1}{\omega_k(B)} \int_B |f(x)|^p w_k(x)dx  \right)^{\frac{1}{p}} \leq \left(\frac{1}{\omega_k(B)} \int_B |f(x)-f_B|^p  w_k(x)dx \right)^{\frac{1}{p}} + |f_B|.
    \end{align*}
    Note that $|f_B| \leq |f|_B$ and $|f|_B \leq \|f\|_{BMO(\mathcal{L}_k)}$ for all $B=B(x_0,r)$ with $r\geq{\rho_k}(x_0)$.
\end{proof}
We establish a characterization of the space $BMO(\mathcal{L}_k)$ that avoids distinguishing between different types of balls. The proof is obtained by adapting the corresponding characterization of the $bmo$ space on metric measure spaces with doubling measures \cite{Dafni-2012}.
\begin{theorem}
    Let $f\in L_{k,\mathrm{loc}}^{1}(\mathbb{R}^n)$. Then $f\in BMO(\mathcal{L}_k)$  if and only if for every ball $B$ in $\mathbb{R}^n$ there exists a constant $\mathfrak{C}_B$ such that the following conditions are satisfied:
    \begin{enumerate}[$(i)$]
     \item 
    \begin{align*}
        C_p = \sup\limits_{B} \left(  \frac{1}{\omega_k(B)} \int_{B}|f(x)-\mathfrak{C}_B|^pw_k(x)dx \right)^{\frac{1}{p}} <\infty,  
    \end{align*}$\text{ where  } 1\leq p< \infty.$
    \item 
    \begin{align*}
C
= \sup_{\substack{B = B(x,r) \\ x \in \mathbb{R}^n,\; r > 0}}
\,\frac{|\mathfrak{C}_B|}{\log_2\!\left(1 + \frac{{\rho_k}(x)}{r}\right)} < \infty.
\end{align*} Moreover, the $BMO(\mathcal{L}_k)$ is related with the constants $C_p$ and $C$ by following relation
\begin{align*}
    \|f\|_{BMO(\mathcal{L}_k)} \approx \inf\limits_{\mathfrak{C}_B} \{\max\{C_p,C \}\},
\end{align*} where the infimum is evaluated over all  $\mathfrak{C}_B$ satisfying the above two conditions. 
    \end{enumerate}
\end{theorem} 

\begin{proof}
    We begin the proof by assuming that $f\in BMO(\mathcal{L}_k)$. It is standard to define the constant $\mathfrak{C}_B$, associated with a ball $B = B(x,r)$, $x \in \mathbb{R}^n$, by
\[
\mathfrak{C}_B =
\begin{cases}
f_B, & r < \rho_k(x),\\
0, & r \geq \rho_k(x).
\end{cases}
\] 
Then Corollary \ref{Corollary J-N} and the fact that $f$ is a $BMO(\mathcal{L}_k)$ function imply condition $(i)$ for each $p\in [1,\infty)$. Also, the choice of $\mathfrak{C}_B$ for larger balls reduces the proof of the condition $(ii)$. It is enough to show that 
\begin{align*}
    \frac{|f_B|}{\log_2 \left(1+ \frac{{\rho_k}(x)}{r} \right)} \leq C\|f\|_{BMO(\mathcal{L}_k)} \, \text{ for } r \leq {\rho_k}(x),
\end{align*} where the constant $C$ is independent of $r$.
We choose $m$ as the smallest positive integer such that, $2^{m}r \geq {\rho_k}(x)$ and consider the collection of balls $B_j=B(x, 2^jr)$ for $j=0,1,\cdots, m$. Using Corollary \ref{Corollary J-N} and \eqref{doubling property}, we deduce that
\begin{align*}
    |f_B|  & \leq \sum_{j=0}^{m-1} |f_{B_j}-f_{B_{j+1}}|+|f_{B_m}|\\
    & \leq \sum_{j=0}^{m-1} \frac{1}{\omega_k(B_j)} \int_{B_j}|f(y)-f_{B_{j+1}}|w_k(y)dy + \frac{1}{\omega_k(B_m)} \int_{B_m}|f(y)|w_k(y)dx\\
    & \leq \sum_{j=0}^{m-1}  \frac{C_{w_k}}{\omega_k(B_{j+1})} \int_{B_{j+1}}|f(y)-f_{B_{j+1}}|w_k(y)dy + \frac{1}{\omega_k(B_m)} \int_{B_m}|f(y)|w_k(y)dx\\ 
    & \leq \left( mC_{w_k}+1 \right)\|f\|_{BMO(\mathcal{L}_k)}.
\end{align*} 
Since $\frac{\rho_k(x)}{r} \geq 1$, we can find a constant $C>0$, independent of the radius, such that
\[
m \leq \log_2\left( \frac{\rho_k(x)}{r} \right) + 1 \leq C \log_2\left( \frac{\rho_k(x)}{r} + 1 \right).
\]
Thus, we obtain the desired inequality.\\

Now we consider the converse part, assuming that  $f \in L_{k,\text{loc}}^1(\mathbb{R}^n)$ and satisfying conditions $(i)$ and $(ii)$.  Then by condition $(i)$ with $p=1$, there exists a constant $\mathfrak{C}_{B_0}$ for the ball $B_0=B(x_0,r_0)$ such that
\begin{align}\label{Esti-1}
    \frac{1}{\omega_k(B_0)} \int_{B_0} |f(x)-f_{B_0}| w_k(x)dx \leq  \frac{2}{\omega_k(B_0)} \int_{B_0} |f(x)-\mathfrak{C}_{B_0}| w_k(x)dx \leq 2C_1.
\end{align}
In particular, for $r\geq{\rho_k}(x_0)$ we have
\begin{align}\label{Esti-2}
    \frac{1}{\omega_k(B_0)} \int_{B_0} |f(x)| w_k(x)dx \leq   \frac{1}{\omega_k(B_0)} \int_{B_0} |f(x)-\mathfrak{C}_{B_0}| w_k(x)dx \, +\,  |\mathfrak{C}_{B_0}| \leq C_1+ C.
\end{align}
The estimates \eqref{Esti-1} and \eqref{Esti-2} show that the $BMO(\mathcal{L}_k)$ norm of $f$ is the $\max\{2C_1, C_1+ C \}$, which is nothing but the constant multiple of $\max\{C_1, C \}$. 
\end{proof}
\begin{remark}
The local $bmo$ space is first defined in the Euclidean setting \cite{Goldberg} and later extended to more general settings, such as metric measure spaces with a doubling measure \cite{Dafni-2012}. Similarly, we define the $bmo$ space in the Dunkl setting using the doubling measure $\omega_k$. Fix a radius $R > 0$. The space $bmo$ is the collection of all $f\in L_{k,\text{loc}}^1(\mathbb{R}^n)$ such that
\[
\|f\|_{bmo} = \sup_{B}
\frac{1}{\omega_k(B)}
\int_B |f(x) - \mathfrak{C}_B|\, w_k(x)\, dx
< \infty. 
\]
where the supremum is taken over all balls $B=B(x,r) \subset\mathbb{R}^n$, and
$
\mathfrak{C}_B =
\begin{cases}
f_B, & \text{if } r < R, \\
0,   & \text{if } r \ge R.
\end{cases}$
Here $r$ is the radius of the ball. The space $bmo$ can be viewed as a particular case of the space $BMO(\mathcal{L}_k)$
corresponding to a constant potential function $V$. Indeed, let $V=C$ 
for some constant $C>0$. Then the critical radius associated with $V$ is
given by 
\[
\rho_k(x)=\frac{1}{\sqrt{C}}, \qquad \text{for every } x\in\mathbb{R}^n.
\]
Now, if $V\in L_k^1(\mathbb{R}^n)$, then $\rho_k(x)=\infty$ for all
$x\in\mathbb{R}^n$. In this case, the localization disappears and
$BMO(\mathcal{L}_k)$ coincides with $BMO_k$.
\end{remark}

\section{$BMO(\mathcal{L}_k)$ as the dual space of Hardy space} \label{S:4}
In this section, we identify the space $BMO(\mathcal{L}_k)$ as the dual of the Hardy space associated with the operator $\Tilde{\mathcal{L}_k}$, whose potential function satisfies the reverse H\"older inequality. In \cite{Hejna-2021}, the Hardy space $H_{\Tilde{\mathcal{L}}_k}^1$ is defined by
\begin{align*}
    H_{\Tilde{\mathcal{L}}_k}^1 = \left\{ f \in L_k^1(\mathbb{R}^n) : f^*(x) = \sup_{t>0} |\mathbf{K}_t f(x)| \in L_k^1(\mathbb{R}^n) \right\},
\end{align*}
with norm $\|f\|_{H_{\Tilde{\mathcal{L}}_k}^1} = \|f^*\|_{L_k^1(\mathbb{R}^n)}$. Here $\mathbf{K}_t$ is the integral transform associated with the Dunkl-Schr\"odinger operator \eqref{Integral operator}. An atomic characterization of $H_{\Tilde{\mathcal{L}}_k}^1$ was also obtained in terms of atoms satisfying appropriate size and cancellation conditions, analogous to those in the classical Hardy space theory associated with Schr\"odinger operators \cite{Jacek-2004}. Moreover, the cancellation conditions imposed on the atoms are determined by a suitable collection of cubes in $\mathbb{R}^n$. Further results related to the operator $\Tilde{\mathcal{L}}_k$ can be found in \cite{Jacek-2025, Hejna-2021, Hejna-2021_2}. 
\par The available atomic decomposition for $H_{\Tilde{\mathcal{L}}_k}^1$ is formulated in terms of cubes. To establish the duality, we require a characterization of $H_{\Tilde{\mathcal{L}}_k}^1$ in terms of atoms whose cancellation depends on the critical radius (cf.\ \cite{Dziubanski-2005}). An analogous characterization is available for the Hardy space associated with the Dunkl harmonic oscillator \cite{Hejna-2020}. In this section, we also derive an atomic decomposition for $H_{\Tilde{\mathcal{L}}_k}^1$ under the following assumption on the potential function $V$: \\ 
\begin{enumerate}[$(i)$]
    \item For all $N>0$, there exist positive constants $c $ and $ C $ such that for all $x,y\in \mathbb{R}^n$ and $t>0$ 
\begin{align} 
    k_t(x,y) \leq C {\left( 1+ \mathcal{E}_k(x,y)\mathcal{}   \right) ^{-N} } \exp\left(-\frac{d_\mathcal{O}(x,y)^2}{ct} \right) \left(\omega_k(B(x,\sqrt{t}))\right)^{-1}
    , \label{assumption-1} 
\end{align} where
\begin{align*} 
\mathcal{E}_k(x,y)=\frac{t}{\omega_k(B(x,\sqrt{t}))} \int_{B(x,\sqrt{t})} V(\xi)w_k(\xi)d\xi + \frac{t}{\omega_k(B(y,\sqrt{t}))} \int_{B(y,\sqrt{t})} V(\xi)w_k(\xi)d\xi.
\end{align*} 

\end{enumerate}
The proof techniques for the atomic decomposition are inspired by \cite{Dziubanski-1999} and \cite{Jacek-2004}. In \cite{Dziubanski-1999}, the arguments rely on the fundamental solutions of Schr\"odinger operators. However, such fundamental solutions are not available for the Dunkl-Schr\"odinger operators. To overcome this difficulty, we incorporate suitable modifications of the methods developed in \cite{Jacek-2004}. 

\subsection{Atomic decomposition:} We recall the definition of atoms associated with the balls and the atomic Hardy space in the following definition. 
\begin{definition} \label{atom-def}
     Let $a$ be a measurable function on $\mathbb{R}^n$. Then $a$ is called an $H_{\Tilde{\mathcal{L}}_k}^1$-atom associated with the ball $B(x_0,r_0)$ if
     \begin{enumerate}[$(i)$]
         \item supp($a) \subset B(x_0,r_0)$,\\
         \item $\|a\|_{L^\infty(\mathbb{R}^n)} \leq \frac{1}{\omega_k(B(x_0,r_0))},$\\
         \item If $x_0 \in \digamma_j$ and  
$0 < r_0 < 2^{-1-\frac{j}{2}}$, then $a$ satisfies the cancellation condition $\int_{\mathbb{R}^n} a(x) w_k(x)\, dx = 0$. 
     \end{enumerate}
      \end{definition}
$H_{\Tilde{\mathcal{L}}_k}^{1,\text{atm}}$ is the space of functions $f\in L_k^1(\mathbb{R}^n)$ which admits a representation of the form 
\begin{align*}
    f=\sum_{m=1}^\infty c_ma_m,
\end{align*} where $c_m \in \mathbb{C}$ and $a_m$ are atoms associated with $H^1_{\Tilde{\mathcal{L}}_k}$ such that $\sum_{m=1}^\infty|c_m|< \infty$. The space $H_{\Tilde{\mathcal{L}}_k}^{1,\text{atm}}$ is a Banach space with the norm 
\begin{align*}
    \|f\|_{H_{\Tilde{\mathcal{L}}_k}^{1,\text{atm}}} = \inf \left\{ \sum\limits_{m=1}^\infty |c_j| : f= \sum\limits_{m=1}^\infty c_ja_j \text{ where, }c_j\in \mathbb{C} \text{ and } 
    a_j \text{ are } H_{\Tilde{\mathcal{L}}_k}^{1} \text{-atoms} \right \}.
\end{align*}
      
The following theorem explains the atomic decomposition of $H_{\Tilde{\mathcal{L}}_k}^1$ functions.

\begin{theorem} \label{atomic decomposition} Assume that $V$ satisfies \eqref{assumption-1}. Then there exists a constant $C>0$ such that, for every $f\in L_k^1(\mathbb{R}^n)$, we have
\begin{align*}
C^{-1}\|f\|_{H_{\Tilde{\mathcal{L}}_k}^{1,\text{atm}}} \leq \|f\|_{H_{\Tilde{\mathcal{L}}_k}^1} \leq C\|f\|_{H_{\Tilde{\mathcal{L}}_k}^{1,\text{atm}}}. 
\end{align*}
\end{theorem}
We prove the theorem by showing the containment $H_{\Tilde{\mathcal{L}}_k}^{1,\text{atm}} \subset H_{\Tilde{\mathcal{L}}_k}^1$ and $ H_{\Tilde{\mathcal{L}}_k}^1 \subset  H_{\Tilde{\mathcal{L}}_k}^{1,\text{atm}}$. The proof is based on the local Hardy space \cite{Dziubanski-1999, Jacek-2004}. For more details on the local Hardy space, readers can refer to \cite{Goldberg} for the Euclidean setting and \cite{Hejna-2020} for the Dunkl setting.
In order to prove the results, we need some lemmas, which are stated below. 

 \begin{lemma}\label{k-kernel-estimate}
     For each $N>0$, there exist constants $c,C>0$ such that 
     \begin{align}\label{k-kernel-estimate}
         k_t(x,y) \leq \frac{C}{\omega_k(B(x,\sqrt{t}))} \exp\left(-\frac{d_\mathcal{O}(x,y)^2}{ct} \right) \left( 1+ \frac{\sqrt{t}}{\rho_k(x)}+ \frac{\sqrt{t}}{\rho_k(y)} \right)^{-N} 
     \end{align} for all $x,y\in \mathbb{R}^n$ and $t>0$.
 \end{lemma}
\begin{proof}
   Invoking the symmetry of the kernel $k_t$, \eqref{kernel estimate_1} and, \eqref{heat kernel esti_2}, it is enough to prove 
    \begin{align*}
        k_t(x,y) \leq \frac{C}{\omega_k(B(x,\sqrt{t}))} \exp\left(-\frac{d_\mathcal{O}(x,y)^2}{ct} \right) \left( 1+ \frac{\sqrt{t}}{\rho_k(x)} \right)^{-N}, 
    \end{align*} for $\sqrt{t} \leq \rho_k(x)$. Using  \cite[Lemma 7]{Hejna-2021}, we have 
    \begin{align*}
        1 &= \frac{\rho_k(x)^2}{\omega_k(B(x,\sqrt{t}))}\int_{B(x,\rho_k(x))} V(\xi)w_k(\xi)d\xi\\
       & \leq C \left(\frac{\rho_k(x)}{\sqrt{t}}\right)^{2-\frac{n+2\gamma}{q}} \frac{t}{\omega_k(B(x,\sqrt{t}))}\int_{B(x,\sqrt{t})} V(\xi)w_k(\xi)d\xi.
    \end{align*} Thus, we get 
    \begin{align} \label{V-estimate}
      \frac{t}{\omega_k(B(x,\sqrt{t}))}\int_{B(x,\sqrt{t})} V(\xi)w_k(\xi)d\xi \geq   \frac{1}{C }  \left(\frac{\sqrt{t}}{\rho_k(x)}\right)^{2-\frac{n+2\gamma}{q}}.  
    \end{align} Applying \eqref{V-estimate} in \eqref{assumption-1}, we deduce that 
    \begin{align*}
        k_t(x,y) \leq \frac{C}{\omega_k(B(x,\sqrt{t}))} \exp\left(-\frac{d_\mathcal{O}(x,y)^2}{ct} \right) \left( 1+ \frac{\sqrt{t}}{\rho_k(x)} \right)^{-\left({2-\frac{n+2\gamma}{q}}\right)m}, 
    \end{align*} true for all $m>0$. This concludes the desired inequality.
\end{proof}
Before stating the next result, we recall the measure $\mu_V$ associated with the potential function $V$ on $\mathbb{R}^n$ introduced in \cite{Hejna-2021}. For every measurable set $A\subset \mathbb{R}^n$,
\[
\mu_V(A)=\int_A V(x)\,w_k(x)\,dx.
\]
It was shown in \cite{Hejna-2021} that there exists a constant $C_{\mu_V}>0$ such that
\begin{align}\label{doubling measure V}
\mu_V(B(x,2r))
\leq
C_{\mu_V}\mu_V(B(x,r)),
\qquad x\in\mathbb{R}^n,\; r>0.
\end{align}
Further properties of the measure $\mu_V$ can also be found in \cite{Hejna-2021}.
\begin{lemma}
There exists a constant $C>0$ such that for every $x\in B_{(j,d)}$ and 
$\sqrt{t}\leq c_0 \rho_k(x)$,
\begin{align}\label{assumption-2}
\int_{B_{(j,d)}^{**}} \mathcal{G}_{t}(x,y)V(y)w_k(y)dy
\leq
Ct^{-1}
\left(\frac{\sqrt{t}}{\rho_k(x)}\right)^{2-\frac{n+2\gamma}{q}} .
\end{align}
\end{lemma}
\begin{proof}
Using the fact that $V\in RH_k^q(\mathbb{R}^n)$, and H\"olders inequality, we obtain 
\begin{align*}
   I&=\int_{B_{(j,d)}^{**}} \mathcal{G}_{t}(x,y)\,V(y)\,w_k(y)dy \\
    & \leq \left( \frac{1}{\omega_k\left( B_{(j,d)}^{**}\right)} \int_{B_{(j,d)}^{**}}V(y)^qw_k(y)dy  \right)^{1/q} \omega_k\left( B_{(j,d)}^{**}\right)^{1/q} \left(\int_{B_{(j,d)}^{**}} \mathcal{G}_{t}(x,y)^{q^\prime}w_k(y)dy\right)^{1/q^{\prime}},
\end{align*} where $q^\prime$ is the conjugate exponent of $q$. Observe that for each $x\in B_{(j,d)}$, we have $\omega_k(B(x,\rho_k(x)))\sim \omega_k\left( B_{(j,d)}^{**} \right)$. Then 
\begin{align*}
   I_1 &= \omega_k\left( B_{(j,d)}^{**}\right)^{1/q}\left(\int_{B_{(j,d)}^{**}} \mathcal{G}_{t}(x,y)^{q^\prime}w_k(y)dy\right)^{1/q^{\prime}}
    \\ & \leq C  \left(\int_{B_{(j,d)}^{**}} 
    \frac{\omega_k \left( B(x, \rho_k(x)) \right)^{q^\prime/q}}{ \omega_k \left( B(x, \sqrt{t})) \right)^{q^\prime-1}} \frac{ \exp \left( \frac{-q^\prime d_{\mathcal{O}}(x,y)^2}{t} \right) }{ \omega_k \left( B(y, \sqrt{t})) \right) }
    w_k(y)dy\right)^{1/q^{\prime}}.
\end{align*}
Since $\sqrt{t} \leq c_0\rho_k(x)$, it follows from \eqref{Volume ratio} that $I_1 \leq C \left( \frac{\rho_k(x)}{\sqrt{t}} \right)^{(n+2\gamma)/q}$. Using \eqref{RVC} and \eqref{doubling measure V}, we deduce that  
\begin{align*}
    \left( \frac{1}{\omega_k\left( B_{(j,d)}^{**}\right)} \int_{B_{(j,d)}^{**}}V(y)^qw_k(y)dy  \right)^{1/q}  & \leq  \frac{C_{RH}}{\omega_k\left( B_{(j,d)}^{**}\right)} \int_{B_{(j,d)}^{**}}V(y)w_k(y)dy  \\
   & \leq \frac{C}{\rho_k(x)^2}.
\end{align*} Thus, we have the desired inequality
\begin{align*}
    \int_{B_{(j,d)}^{**}} \mathcal{G}_{t}(x,y)V(y)w_k(y)dy
\leq
Ct^{-1}
\left(\frac{\sqrt{t}}{\rho_k(x)}\right)^{2-\frac{n+2\gamma}{q}} .
\end{align*}
\end{proof}

\begin{lemma} \label{lemma-4.1}
    Let $f\in L_k^1(\mathbb{R}^n)$. Then there exists a constant $C>0$ such that for each $\Psi_{(j,d)}$,  we have 
    \begin{align} \label{Lemma4.4-1}
        \int_{\mathbb{R}^n \setminus B_{(j,d)^*}} \,\sup\limits_{0<t\leq 2^{-j}} H_t(|\Psi_{(j,d)} f|)(x)w_k(x)dx \leq C \|\Psi_{(j,d)} f\|_{L_k^2(\mathbb{R}^n)}
    \end{align} and 
    \begin{align}\label{Lemma4.4-2}
        \int_{\mathbb{R}^n \setminus B_{(j,d)^*}} \,\sup\limits_{0<t\leq 2^{-j}} \mathbf{K}_t(|\Psi_{(j,d)} f|)(x)w_k(x)dx \leq C \|\Psi_{(j,d)} f\|_{L_k^2(\mathbb{R}^n)}.
    \end{align}
\end{lemma}
\begin{proof}
  We write 
  \begin{align*}
  I& = \int\limits_{\mathbb{R}^n \setminus B_{(j,d)}^*}  \,\sup\limits_{0<t\leq 2^{-j}} H_t(|\Psi_{(j,d)} f|)(x)w_k(x)dx \\ & = \sum_{m=0}^\infty  \int_{\mathbb{R}^n \setminus B_{(j,d)}^*}\,\sup\limits_{2^{-(j+m+1)}<t\leq 2^{-(j+m)}} \int\limits_{B(x_{(j,d)}, 2^{1-j/2})} h_t(x,y) \Psi_{(j,d)}(y) | f(y)| w_k(y)dyw_k(x)dx .
  \end{align*}  Note that $x\in \mathbb{R}^n \setminus B_{(j,d)}^* $ and $y\in B(x_{(j,d)}, 2^{1-j/2})$, then we have $|x-y|\geq 2^{-j/2}$. Invoking   \eqref{heat kernel esti_2}, we get 
  \begin{align*}
   I &\leq  \sum_{m=0}^\infty C \frac{2^{-(m+j)}}{2^{-j}} \int\limits_{B(x_{(j,d)}, 2^{1-j/2})}|\Psi_{(j,d)}f(y)| \int\limits_{\mathbb{R}^n \setminus B_{(j,d)}^*}  \frac{\exp\left( -\frac{cd_\mathcal{O}(x,y)^2}{2^{-(j+m)}} \right)}{B(x,2^{-(j+m)/2})} w_k(x)dx w_k(y)dy\\
   & \leq C\|\Psi_{(j,d)}f\|_{L_k^1(\mathbb{R}^n)}.
  \end{align*} Thus, we have the inequality \eqref{Lemma4.4-1}. Now the inequality \eqref{Lemma4.4-2} will follows from \eqref{kernel estimate_1} and \eqref{Lemma4.4-1}.
\end{proof}
\begin{lemma}\label{Lamma B2}
   There exists a constant $C>0$ such that for every $\Psi_{(j,d)}$ and $f\in L_k^1(\mathbb{R}^n)$ we have 
   \begin{align*}
       \Bigg\| \sup\limits_{0<t\leq2^{-j}} \big| (H_t-\mathbf{K}_t) (\Psi_{(j,d)}f)\big| \Bigg\|_{L_k^1(\mathbb{R}^n)} \leq C \|\Psi_{(j,d)}f\|_{L_k^1(\mathbb{R}^n)}.
   \end{align*}
\end{lemma}
\begin{proof}
    Using Lemma \ref{lemma-4.1}, it is enough to prove 
    \begin{align*}
        \Bigg\| \sup\limits_{0<t\leq2^{-j}} \big| (H_t-\mathbf{K}_t) (\Psi_{(j,d)}f)\big| \Bigg\|_{L_k^1(B_{(j,d)})^*} \leq C \|\Psi_{(j,d)}f\|_{L_k^1(\mathbb{R}^n)}.
    \end{align*}
    Invoking the perturbation formula, we have 
    \begin{align*}
        (H_t-\mathbf{K}_t) (\Psi_{(j,d)}f)(x) & = \int_0^t \int_{\mathbb{R}^n} h_{t-s}(x,y)V(y)\mathbf{K}_s(\Psi_{(j,d)}f)(y)w_k(y)dy\,ds\\
        &= \int_0^t \int_{\mathbb{R}^n}  h_{t-s}(x,y)(V_1(y)+V_2(y))\mathbf{K}_s(\Psi_{(j,d)}f)(y)w_k(y)dy\,ds\\
        & =I_1+I_2,
    \end{align*}
where 
\begin{align*}
    I_i(x) = \int_0^t \int_{\mathbb{R}^n}  h_{t-s}(x,y)V_i(y)\mathbf{K}_s(\Psi_{(j,d)}f)(y)w_k(y)dy\,ds \quad \text{ for } i=1,2 
\end{align*}
and $V_1=V\chi_{B_{(j,d)}^*}$ and $V_2=V\chi_{\mathbb{R}^n \setminus B_{(j,d)}^*}$. Note that for each $y\in \mathbb{R}^n \setminus B_{(j,d)}^* $ and $x\in B_{(j,d)}$, we have $|x-y|> 2^{-j/2}$. Consequently, 
\begin{align*}
   &\sup\limits_{0<t\leq 2^{-j}}|I_2(x)| \\&=  \sup\limits_{0<t\leq 2^{-j}}\Bigg| \int_0^t \int_{\mathbb{R}^n} h_{t-s}(x,y)V_2(y)\mathbf{K}_s(\Psi_{(j,d)}f)(y)w_k(y)dy\,ds
   \Bigg|\\
   & \leq \sum_{m=0}^\infty\, \sup\limits_{2^{-(j+m+1)}<t\leq 2^{-(j+m)}}\int_0^t \int_{\mathbb{R}^n} h_{t-s}(x,y)V_2(y)\mathbf{K}_s(|\Psi_{(j,d)}f|)(y)w_k(y)dy\,ds\\
   &  \leq \sum_{m=0}^\infty \,\sup\limits_{2^{-(j+m+1)}<t\leq 2^{-(j+m)}} \sum_{l=0}^\infty  \int_{t-2^{-l}t}^{t-2^{-(l+1)}t} \int_{\mathbb{R}^n} h_{t-s}(x,y)V_2(y)\mathbf{K}_s(|\Psi_{(j,d)}f|)(y)w_k(y)dy\,ds.
 \end{align*} Using \eqref{heat kernel esti_2}, we have
\begin{align*}
\sup_{0<t \leq 2^{-j}} |I_2(x)|
&\leq C \sum_{m,l=0}^\infty 
\sup_{2^{-(j+m+1)} < t \leq 2^{-(j+m)}} 
\int_{t-2^{-l}t}^{t-2^{-(l+1)}t}  
\frac{t-s}{\omega_k(B(x,\sqrt{t-s}))} \\
&\quad \quad \times \int_{\mathbb{R}^n} 
\frac{\exp\left(-\frac{c\, d_\mathcal{O}(x,y)^2}{t-s}\right)}{|x-y|^2} 
\, V_2(y)\, \mathbf{K}_s(|\Psi_{(j,d)}f|)(y)\, w_k(y)dy\, ds \\
&\leq C \sum_{m,l=0}^\infty  
\int_0^\infty \int_{\mathbb{R}^n} 
\frac{2^{-(j+l+m)}}{2^{-j}} 
\frac{\exp\left(-\frac{c\, d_\mathcal{O}(x,y)^2}{2^{-(j+l+m)}}\right)}
{\omega_k(B(x,2^{-(j+l+m)/2}))} \\
&\quad \quad \times V_2(y)\, \mathbf{K}_s(|\Psi_{(j,d)}f|)(y)\, w_k(y)dy\, ds.
\end{align*}
Integrating both sides with respect to $w_k(x)\,dx$ and applying Tonelli's theorem, we obtain
\begin{align*}
    \int_{B_{(j,d)}^*}  \sup\limits_{0<t\leq 2^{-j}}|I_2(x)| w_k(x)dx 
    \leq C \sum_{m,l=0}^{\infty} 2^{-(l+m)} \int_0^\infty \int_{\mathbb{R}^n} V(y) \mathbf{K}_s(|\Psi_{(j,d)}f|)(y)w_k(y)dy\,ds.
\end{align*} Finally, invoking \cite[Lemma 16]{Hejna-2021}, we deduce that 
\begin{align*}
     \int_{B_{(j,d)}^*}  \sup\limits_{0<t\leq 2^{-j}}|I_2(x)| w_k(x)dx  \leq C \|\Psi_{(j,d)}f\|_{L_k^1(\mathbb{R}^n)}.
\end{align*}
Now, considering the case $I_1$, we write 
\begin{align*}
  \int_0^t \int_{\mathbb{R}^n}& h_{t-s}(x,y)V_1(y)\mathbf{K}_s(\Psi_{(j,d)}f)(y)w_k(y)dy\,ds  \\ & = \left(\int_0^{t/2} + \int_{t/2}^t\right)\int_{\mathbb{R}^n} h_{t-s}(x,y)V_1(y)\mathbf{K}_s(\Psi_{(j,d)}f)(y)w_k(y)dy\,ds\\
  & = J_1(x)+ J_2(x).
\end{align*} 
Using \eqref{heat kernel esti_2}, \eqref{kernel estimate_1}, \eqref{assumption-2}, and Fubini's theorem, we obtain 

\begin{align*}
    & \Bigg\| \sup_{0<t\leq 2^{-j}} |J_1(x)|  \Bigg\|_{L_k^1(\mathbb{R}^n)}\\
    & \leq \sum_{m=0}^\infty \Bigg\|  \sup_{2^{-(j+m+1)}<t\leq 2^{-(j+m)}} |J_1(x)|  \Bigg\|_{L_k^1(\mathbb{R}^n)} \\
    & \leq C\sum_{m=0}^\infty \int_{\mathbb{R}^n} \int_0^{2^{-(j+m)}}\int_{\mathbb{R}^n}\frac{\exp\left(-\frac{d_\mathcal{O}(x,y)^2}{c2^{-(j+m)}} \right)}{\omega_k(B(y,2^{-(J+m)/2}))}V_1(y)\mathbf{K}_s(|\Psi_{(j,d)}f|)(y)w_k(y)dy \,ds\,w_k(x)dx\\
    & \leq C \sum_{m=0}^\infty \int_0^{2^{-(j+m)}} \int_{\mathbb{R}^n} V_1(y) \int_{\mathbb{R}^n} k_s(y,z)\Psi_{(j,d)}(z)|f(z)|w_k(z)dz\,w_k(y)dy\,ds \\
    & \leq C \sum_{m=0}^\infty  \int_0^{2^{-(j+m)}} \int_{\mathbb{R}^n} \Psi_{(j,d)}(z)|f(z)| \int_{B_{(j,d)^*}}k_s(y,z)V(y)w_k(y)dy\,w_k(z)dz\,ds\\
    & \leq C \sum_{m=0}^\infty   \int_{\mathbb{R}^n} \Psi_{(j,d)}(z)|f(z)|  \int_0^{2^{-(j+m)}} \int_{B_{(j,d)^*}} \frac{\exp\left( -\frac{d_\mathcal{O}(z,y)^2}{cs} \right)}{\max\left\{ \omega_k(B(z,\sqrt{s})),\omega_k(B(y,\sqrt{s}) )\right\}}\\ \\
    & \quad \quad \quad \times V(y)w_k(y)dy  \, ds\,w_k(z)dz \\ \\
    & \leq C  \sum_{m=0}^\infty  \int_{\mathbb{R}^n} \Psi_{(j,d)}(z)|f(z)|  \int_0^{2^{-(j+m)}} \left( \frac{\sqrt{s}}{\rho_k(z)} \right)^{2-(n+2\gamma)/q} \frac{ds}{s} w_k(z) dz
\end{align*}
Using the relation that $\rho_k(z) \sim 2^{-j/2}$ for $z\in B_{(j,d)}^*$ and the condition $2-(n+2\gamma)/q >0$, we get
\begin{align*}
\Bigg\| \sup_{0<t\leq 2^{-j}} |J_1(x)|  \Bigg\|_{L_k^1(\mathbb{R}^n)}    & \leq C  \sum_{m=0}^\infty 2^{-j(1-(n+2\gamma)/4q)} \int_{\mathbb{R}^n} \Psi_{(j,d)}(z)|f(z)| w_k(z)dz\\
    & \leq C\,\| \Psi_{(j,d)}f\|_{L_k^1(\mathbb{R}^n)}.
\end{align*}  
In order to get the estimate for $J_2$, we write
\begin{align*}
    \Bigg\| \sup_{0<t\leq 2^{-j}} |J_2(x)|  \Bigg\|_{L_k^1(\mathbb{R}^n)}  \leq \sum_{m=0}^\infty  \Bigg\| \sup_{2^{-(j+m+1)}<t\leq 2^{-(j+m)}} |J_2(x)|  \Bigg\|_{L_k^1(\mathbb{R}^n)}.
\end{align*} Applying a change of variable, we have
\begin{align*}
    |J_2(x)| \leq \int_{\mathbb{R}^n}\int_0^{t/2}  h_s(x,y)V_1(y)\mathbf{K}_{(t-s)}(|\Psi_{(j,d)}f|)(y)w_k(y)dy\,ds.
\end{align*} Employing \eqref{heat kernel esti_2} and \eqref{kernel estimate_1}, we obtain 
\begin{align}
   &  \sup \limits_{2^{-(j+m+1)} < t\leq 2^{-(j+m)}} |J_2(x)|\notag\\ 
   & \leq {C} \int\limits_{\mathbb{R}^n} \int\limits_{0}^{2^{-(j+m+1)}} \mathcal{G}_{s/c}(x,y)\mathcal{G}_{2^{-(j+m)/c}}(y,z) V_1(y)\Psi_{(j,d)}(z)|f(z)|w_k(y)dy\,ds\, w_k(z)dz. \label{J_2-estimate}
\end{align} For $s \leq t \leq 2^{-(j+m+1)}$, we obtain 
\begin{align*}
   \exp\left( -\frac{cd_\mathcal{O}(x,y)^2}{s} \right)&  \exp\left( -\frac{cd_\mathcal{O}(y,z)^2}{2^{-(j+m)}} \right) \\& \leq  \exp\left( -\frac{cd_\mathcal{O}(x,y)^2}{2s} \right) \exp\left( -\frac{cd_\mathcal{O}(x,y)^2}{2^{-(j+m)}} \right) \exp\left( -\frac{cd_\mathcal{O}(y,z)^2}{2^{-(j+m)}} \right)\\
   & \leq \exp\left( -\frac{cd_\mathcal{O}(x,y)^2}{2s} \right)  \exp \left(  -\frac{cd_\mathcal{O}(x,z)^2}{2^{-(j+m+1)}} \right).
\end{align*}
Consequently, \eqref{J_2-estimate} can be bounded by 
\begin{align*}
     &  \sup \limits_{2^{-(j+m+1)} < t\leq 2^{-(j+m)}} |J_2(x)|\\ 
     & \leq {C}\int\limits_{\mathbb{R}^n} \int\limits_0^{2^{-(j+m+1)}}  \int\limits_{\mathbb{R}^n}\mathcal{G}_{2s/c}(x,y) \frac{\exp\left( -\frac{cd_\mathcal{O}(x,z)^2}{2^{-(j+m+1)}} \right)}{\omega_k(B(z,2^{-(j+m)/2}))}V_1(y) \Psi_{(j,d)}(z)|f(z)|w_k(y)dy\,ds\,w_k(z)dz.
\end{align*} Now, we use arguments similar to those in the case of $V_2$. Using Tonelli's theorem and \eqref{assumption-2}, we deduce that
\begin{align*}
       \sup \limits_{2^{-(j+m+1)} < t\leq 2^{-(j+m)}} |J_2(x)|\leq C 2^{-j(1-(n+2\gamma)/4q)} \|\Psi_{(j,d)}f\|_{L_k^1(\mathbb{R}^n)}.
\end{align*} Hence, 
\begin{align*}
     \Bigg\| \sup_{0<t\leq 2^{-j}} |J_2(x)|  \Bigg\|_{L_k^1(\mathbb{R}^n)}  \leq C\|\Psi_{(j,d)}f\|_{L_k^1(\mathbb{R}^n)}.
\end{align*}
\end{proof}
 \begin{lemma}\label{Lemma B3}
     Let $$\mathcal{M}_{(j,d)}f(x)= \sup_{0<t\leq 2^{-j}}\Big| \mathbf{K}_t\left(\Psi_{(j,d)}f\right)(x)- \Psi_{(j,d)}(x)\mathbf{K}_t(f)(x)\Big|.$$
     Then there exists a constant $C>0$ such that for all $f\in H_{\Tilde{\mathcal{L}}_k}^1$ we have 
     \begin{align*}
\sum_{(j,d)}\|\mathcal{M}_{(j,d)}f\|_{L_k^1(\mathbb{R}^n)} \leq C \left(\|f\|_{L_k^1(\mathbb{R}^n)} +\|f\|_{H_{\Tilde{\mathcal{L}}_k}^1}\right) .
     \end{align*}  
 \end{lemma}
\begin{proof}
We write 
\begin{align*}
   & \|\mathcal{M}_{(j,d)}f\|_{L_k^1(\mathbb{R}^n)} \leq \|\mathcal{M}_{(j,d)}f\|_{L_k^1\left(B_{(j,d)}^{*}\right)} + \,\|\mathcal{M}_{(j,d)}f\|_{L_k^1\left(\mathbb{R}^n \setminus B_{(j,d)}^{*}\right)} \\
    & \leq \|\mathcal{M}_{(j,d)}f\|_{L_k^1\left(B_{(j,d)}^{*}\right)} + \Bigg\| \sup_{0<t\leq 2^{-j}}\Big| \mathbf{K}_t\left(\Psi_{(j,d)}f\right) \Big| \Bigg\|_{L_k^1\left(\mathbb{R}^n \setminus B_{(j,d)}^{*}\right)} + \Bigg\| \Psi_{(j,d)} \sup_{0<t\leq 2^{-j}}\Big|\mathbf{K}_t(f) \Big|\Bigg\|_{L_k^1(\mathbb{R}^n)}.
\end{align*} Using the Lemma \ref{Lamma B2} and the fact that $f\in H_{\Tilde{\mathcal{L}}_k}^1$, we can deduce the result by proving the following 
   \begin{align*}
\sum_{(j,d)}\|\mathcal{M}_{(j,d)}f\|_{L_k^1(\mathbb{R}^n)} \leq C \|f\|_{L_k^1\left(B_{(j,d)}^{*}\right)}.
\end{align*}
 Now, we consider 
 \begin{align*}
    [\Psi_{(j,d)},\mathbf{K}_t] f(x) &=  [\Psi_{(j,d)},\mathbf{K}_t] \left(\sum_{(j^\prime, d^\prime)}\Psi_{(j^\prime,d^\prime)}f\right)(x)
    =  \sum_{(j^\prime, d^\prime)} \mathcal{P}_{t,(j,d),(j^\prime, d^\prime)}f(x),
 \end{align*}
 where 
\begin{align*}
    \mathcal{P}_{t,(j,d),(j^\prime, d^\prime)}f(x) 
    & = \int_{\mathbb{R}^n}k_t(x,y)f(y) \left( \Psi_{(j,d)}(x)-\psi_{(j,d)}(y) \right) \Psi_{(j^\prime,d^\prime)}(y)w_k(y)dy.
\end{align*}
We define the sets 
\begin{align*}
    \mathcal{I}_{(j,d)} = \{(j^\prime, d^\prime): |x_{(j,d)} - x_{(j^\prime,d^\prime)}| \leq c2^{-j/2}  \}
\end{align*}
and \begin{align*}
    \mathcal{J}_{(j,d)}=\{(j^\prime, d^\prime): |x_{(j,d)} - x_{(j^\prime,d^\prime)}| > c2^{-j/2}  \}.
\end{align*}
By Proposition \ref{cover_1}, the cardinality of the set $\mathcal{I}_{(j,d)}$ is bounded by a constant that is independent of $(j,d)$. Furthermore, we choose sufficiently large ${c}$  so that the balls $B_{(j,d)}^{**}$ and $B_{(j^{\prime},d^{\prime})}^{**}$ are disjoint whenever $(j^{\prime},d^{\prime}) \in \mathcal{J}_{(j,d)}$. Moreover, if $(j^{\prime},d^{\prime}) \in \mathcal{J}_{(j,d)}$, then
\[
|x- y| \sim |x_{(j,d)}, x_{(j^{\prime},d^{\prime})}|
\quad \text{for all } x \in B_{(j,d)} \text{ and } y \in B_{(j^{\prime},d^{\prime})}.
\] We denote $f_{\mathcal{I}_{(j,d)}}(x)= \sum_{(j^{\prime},d^{\prime}) \in \mathcal{I}_{(j,d)}} \Psi_{(j^\prime,d^\prime)}f(x)$  and  $   f_{\mathcal{J}_{(j,d)}}(x)= \sum_{(j^{\prime},d^{\prime}) \in \mathcal{J}_{(j,d)}} \Psi_{(j^\prime,d^\prime)}f(x).$  
Then $ f = f_{\mathcal{I}_{(j,d)}}+ f_{\mathcal{J}_{(j,d)}}$. \\
Using \eqref{heat kernel esti_2},  \eqref{kernel estimate_1}, and the properties of $\Psi_{(j,d)}$,  for $(j^\prime,d^\prime) \in  \mathcal{I}_(j,d)$, we observe that
\begin{align*}
   &\sup\limits_{0<t\leq 2^{-j}} \Big| \mathcal{P}_{t,(j,d),(j^\prime, d^\prime)}f(x)\Big| 
   \\ &\leq  \sup\limits_{0<t\leq 2^{-j}}  \int_{\mathbb{R}^n} k_t(x,y)|f(y) \left( \Psi_{(j,d)}(x)-\Psi_{(j,d)}(y) \right)| \Psi_{(j^\prime,d^\prime)}(y)w_k(y)dy\\
   & \leq C \sup\limits_{0<t\leq 2^{-j}} \int_{\mathbb{R}^n}|x-y|2^{j/2} \frac{\sqrt{t}}{|x-y|} \frac{\exp\left(-\frac{cd_\mathcal{O}(x,y)^2}{t} \right)}{\omega_k\left( B(x, \sqrt{t} \right)}\Psi_{(j^\prime,d^\prime)}(y)|f(y)|w_k(y)dy\\
   &\leq C \sum_{m=0}^\infty \sup\limits_{2^{-(j+m+1)}<t\leq 2^{-(j+m)}} \int\limits_{\mathbb{R}^n} \frac{2^{-(j+m)/2}}{{2^{-j/2}}} \frac{\exp\left(-\frac{cd_\mathcal{O}(x,y)^2}{2^{-(j+m)}} \right)}{\omega_k\left( B(x, 2^{-(j+m)/2} \right)}\Psi_{(j^\prime,d^\prime)}(y)|f(y)|w_k(y)dy.
\end{align*} 
Integrating over $B_{(j,d)}^{*}$ with respect to the variable $x$, and Fubinis theorem gives  
\begin{align*}
    \Bigg\| \sup\limits_{0<t\leq 2^{-j}} \Big| \mathcal{P}_{t,(j,d),(j^\prime, d^\prime)}f(x)\Big|    \Bigg\|_{L_k^1\left(B_{(j,d)}^{*}\right)} \leq C \|\Psi_{(j^\prime,d^\prime)}f\|_{L_k^1(\mathbb{R}^n)}.
\end{align*} Since $(j^\prime, d^\prime)\in \mathcal{I}_{(j,d)}$, we have
$$ \Bigg\| \sup_{0<t\leq 2^{-j}}[\Psi_{(j,d)},\mathbf{K}_t]f_{\mathcal{I}_{(j,d)}}\Bigg\|_{L_k^1\left(B_{(j,d)}^{*}\right)} \leq C \|f\|_{L_k^1(B_{(j,d)})},$$ and consequently 
\begin{align*}
\sum_{(j,d)}\|\mathcal{M}_{(j,d)}f_{\mathcal{I}_{(j,d)}}\|_{L_k^1(\mathbb{R}^n)} \leq C\|f\|_{L_k^1(\mathbb{R}^n)}.
\end{align*}
On the other hand, for  $(j^\prime,d^\prime) \in  \mathcal{J}_{(j,d)}$, one can note that $\Psi_{(j,d)}(y)=0$ for $y\in B_{(j^\prime,d^\prime)}$. Thus,
\begin{align*}
     \Bigg\|&\sup\limits_{0<t\leq 2^{-j}}  \mathcal{P}_{t,(j,d),(j^\prime, d^\prime)}f(x)\Bigg\|_{L_k^1\left( B_{(j,d)}^* \right)}\\ 
     &\leq \int_{\mathbb{R}^n} \chi_{B_{(j,d)}^* } (x)\sup\limits_{0<t\leq 2^{-j}}  \int_{\mathbb{R}^n} k_t(x,y)|f(y)| \Psi_{(j^\prime,d^\prime)}(y)w_k(y)dy\,w_k(x)dx.
\end{align*} Therefore, 
\begin{align*}
\sum_{(j,d)}\Big\|\mathcal{M}_{(j,d)}&f_{\mathcal{J}_{(j,d)}}\Big\|_{L_k^1\left(B_{(j,d)}^*\right)} \leq 
\sum_{(j,d)} \sum_{(j^\prime,d^\prime) \in  \mathcal{J}_{(j,d)}} \Big \| \chi_{B_{(j,d)}^* } \sup_{0<t\leq 2^{-j}} \mathbf{K}_t\left(\Psi_{(j^\prime,d^\prime)}|f|\right)\Big\|_{L_k^1(\mathbb{R}^n)}\\
& \leq \sum_{(j,d)} \sum_{(j^\prime,d^\prime) \in  \mathcal{J}_{(j,d)}} \Big \| \chi_{B_{(j,d)}^* } \sup_{t>0} \mathbf{K}_t\left(\Psi_{(j^\prime,d^\prime)}|f|\right)\Big\|_{L_k^1(\mathbb{R}^n)}\\
& \leq \sum_{(j^\prime,d^\prime)} \sum_{(j,d) \in \mathcal{J}_{(j^\prime,d^\prime)} } \Big \| \chi_{B_{(j,d)}^* } \sup_{t>0} \mathbf{K}_t\left(\Psi_{(j^\prime,d^\prime)}|f|\right)\Big\|_{L_k^1(\mathbb{R}^n)}\\
& \leq C \sum_{(j^\prime,d^\prime)} \Big \|  \sup_{t>0} \mathbf{K}_t\left(\Psi_{(j^\prime,d^\prime)}|f|\right)\Big\|_{L_k^1\left(\mathbb{R}^n \setminus B_{(j^\prime,d^\prime)}^*\right )}\\
& \leq C \sum_{(j^\prime,d^\prime)} \Big \|  \sup_{0<t< 2^{-j^\prime}} \mathbf{K}_t\left(\Psi_{(j^\prime,d^\prime)}|f|\right)\Big\|_{L_k^1\left(\mathbb{R}^n \setminus B_{(j^\prime,d^\prime)}^*\right )} \\
& \quad \quad \quad + C \sum_{(j^\prime,d^\prime)} \sum_{m=0} ^ \infty \Big \|  \sup_{2^{-j^\prime+m}<t\leq 2^{-j^\prime+m+1}} \mathbf{K}_t\left(\Psi_{(j^\prime,d^\prime)}|f|\right)\Big\|_{L_k^1\left(\mathbb{R}^n \setminus B_{(j^\prime,d^\prime)}^*\right )} \\
& = \mathcal{J}_1+ \mathcal{J}_2.
\end{align*}
Using Proposition \ref{cover_1}, and Lemma \ref{lemma-4.1} we have $ \mathcal{J}_1 \leq C\|f\|_{L_k^1(\mathbb{R}^n)}$.  To find the estimate for $\mathcal{J}_2$, we use the semigroup property and the bounds of the kernel \eqref{k-kernel-estimate} and \eqref{heat kernel esti_2}. For $2^{-j^\prime+m} \leq t < 2^{-j^\prime+m+1}$, we obtain 
\begin{align*}
    \int_{\mathbb{R}^n} & k_t(x,y)\Psi_{(j^\prime, d^\prime)}(y)|f(y)|w_k(y)dy \\
    & =  \int_{\mathbb{R}^n}  \int_{\mathbb{R}^n} k_{t-2^{-j^\prime+m-1}}(x,z) k_{2^{-j^\prime+m-1}}(z,y)w_k(z)dz\, \Psi_{(j^\prime,d^\prime)}(y)|f(y)|w_k(y)dy\\
    & \leq C \int_{\mathbb{R}^n}  \int_{\mathbb{R}^n} \frac{\exp \left(-\frac{cd_\mathcal{O}(x,y)^2}{2^{-j^\prime+m-1}} \right)}{\omega_k \left(B(x,2^{(-j^\prime+m)/2})  \right)} k_{2^{-j^\prime+m-1}}(z,y)w_k(z)dz\, \Psi_{(j^\prime,d^\prime)}(y)|f(y)|w_k(y)dy.
\end{align*} 
Integrating with respect to $x$ over the set $\mathbb{R}^n \setminus B_{(j^\prime, d^\prime)}^*$ and using \eqref{k-kernel-estimate} yields
\begin{align*}
    \mathcal{J}_2  & \leq C \sum_{m=0}^\infty \sum_{(j^\prime, d^\prime)} \int_{\mathbb{R}^n} \Psi_{(j^\prime, d^\prime)}(y)|f(y)|\int_{\mathbb{R}^n} k_{2^{-j^\prime+m-1}}(z,y)w_k(z)dz\, w_k(y)dy \\
    & \leq C \sum_{m=0}^\infty \sum_{(j^\prime, d^\prime)} \int_{\mathbb{R}^n} \Psi_{(j^\prime, d^\prime)}(y)|f(y)|\int_{\mathbb{R}^n}  \frac{\exp\left(-\frac{d_\mathcal{O}(z,y)^2}{c2^{-j^\prime+m-1}} \right)}{\omega_k(B(z,2^{(-j^\prime+m-1)/2})} \\
    & \quad \quad \quad \quad\times \left(  \frac{2^{(-j^\prime+m-1)/2}}{\rho_k(y)} \right)^{-N}w_k(z)dz\, w_k(y)dy. 
\end{align*} Note that supp$\left(\Psi_{(j^\prime, d^\prime)}\right) \subset B_{(j^\prime, d^\prime)}$, then  $\rho_k(y) \sim \rho_k(x_{(j^\prime,d^\prime)})$ for all $y\in B_{(j^\prime, d^\prime)}$. Then the above integral boils down to 
\begin{align*}
     \mathcal{J}_2 \leq C\|f\|_{L_k^1(\mathbb{R}^n)}.
\end{align*}
\end{proof}

\begin{proof}[Proof of Theorem \ref{atomic decomposition}]
    We assume that $f\in H^1_{\Tilde{\mathcal{L}}_k}$. Using Lemma \ref{Lamma B2}, we obtain  
    \begin{align*}
        &\Bigg\|  \sup_{0<t \leq 2^{-j}}  \big| H_t\left( 
        \Psi_{(j,d)}f  \right)\big|  \Bigg\|_{L_k^1(\mathbb{R}^n)} \\ & \leq \Bigg\| \sup_{0<t \leq 2^{-j}}  \big| (H_t-\mathbf{K}_t)\left( 
        \Psi_{(j,d)}f  \right)\big|  \Bigg\|_{L_k^1(\mathbb{R}^n)} +\, \Bigg\| \sup_{0<t \leq 2^{-j}}  \big| \mathbf{K}_t\left( 
        \Psi_{(j,d)}f  \right)\big|  \Bigg\|_{L_k^1(\mathbb{R}^n)}\\
        & \leq C \|\Psi_{(j,d)}f\|_{L_k^1(\mathbb{R}^n)} + \Bigg\| \sup_{0<t \leq 2^{-j}}  \big| \mathbf{K}_t\left( 
        \Psi_{(j,d)}f  \right) -\Psi_{(j,d)}\mathbf{K}_t(f)\big|  \Bigg\|_{L_k^1(\mathbb{R}^n)} +   \Bigg\| \sup_{0<t \leq 2^{-j}}\Psi_{(j,d)}\mathbf{K}_t(f)\Bigg\|_{L_k^1(\mathbb{R}^n)}.
    \end{align*} By Lemma \ref{Lemma B3} we conclude that
   \begin{align*}
     \sum_{(j,d)}  \Bigg\| & \sup_{0<t \leq 2^{-j}}  \big| H_t\left( 
        \Psi_{(j,d)}f  \right)\big|  \Bigg\|_{L_k^1(\mathbb{R}^n)} \leq C\left( \|f\|_{L_k^1(\mathbb{R}^n)} + \|f\|_{H_{\Tilde{\mathcal{L}}_k}^1} \right).
   \end{align*}
Therefore,  $\Psi_{(j,d)}f\in H_{k,\text{loc}}^1$ for each $(j,d)$. Then by \cite[Proposition 4]{Hejna-2021}, we write 
    \begin{align*}
        \Psi_{(j,d)}f= \sum_{l=1} ^ \infty C^{(j,d)}_l a_l^{(j,d)}, \quad \text{ with } \sum_{l=1}^ \infty \Big|C^{(j,d)}_l\Big| \leq \Bigg\|  \sup_{0<t \leq 2^{-j}}  \big| H_t\left( 
        \Psi_{(j,d)}f  \right)\big|  \Bigg\|_{L_k^1(\mathbb{R}^n)}.
    \end{align*}  Consequently, we have 
    \begin{align*}
        f = \sum_{(j,d)}\Psi_{(j,d)}f = \sum_{(j,d)}  \sum_{l=1}^ \infty C^{(j,d)}_l a_l^{(j,d)}, \quad \text{and } \sum_{(j,d)}  \sum_{l=1}^ \infty \big|C^{(j,d)}_l\big| \leq  C \|f\|_{H^1_{\Tilde{\mathcal{L}}_k}}.
    \end{align*}  Thus, we get  
    $$  \|f\|_{H_{\Tilde{\mathcal{L}}_k}^{1,\text{atm}}} \leq C \|f\|_{H^1_{\Tilde{\mathcal{L}}_k}}.$$
    On the other hand, it is enough to show the existence of a constant $C>0$, so that 
    \begin{align*}
        \|a\|_{H^1_{\Tilde{\mathcal{L}}_k}} \leq C \quad \text{ for all $H_{\Tilde{\mathcal{L}}_k}^{1,\text{atm}}$-atoms.}
    \end{align*} 
    Let $a$ be an 
    $H_{\Tilde{\mathcal{L}}_k}^{1} \text{-atoms}$ associated with the ball $B(x_0,r_0)$. Then by Definition \ref{atom-def} and \cite[Definition 5]{Hejna-2021}, we obtain 
    $$\Bigg\|  \sup\limits_{0<t \leq 2^{-j}}  \big| H_t\left( 
        a  \right)\big|  \Bigg\|_{L_k^1(\mathbb{R}^n)} \leq C.$$ 
    Using Lemma \ref{Lamma B2}, we deduce that 
    $$ \Bigg\|  \sup\limits_{0<t \leq 2^{-j}}  \big| \mathbf{K}_t\left( 
        a  \right)\big|  \Bigg\|_{L_k^1(\mathbb{R}^n)} \leq C.$$ 
    Now it remains to prove that 
     \begin{align*}
        \Bigg\|  \sup_{t> 2^{-j}}  \big| \mathbf{K}_t\left( 
        a  \right)\big|  \Bigg\|_{L_k^1(\mathbb{R}^n)} \leq C. 
     \end{align*} We write 
     \begin{align*}
         \Bigg\|  \sup_{t> 2^{-j}}  \big| \mathbf{K}_t\left( 
        a  \right)\big|  \Bigg\|_{L_k^1(\mathbb{R}^n)} \leq \sum_{m=0}^\infty \Bigg\|  \sup_{ 2^{-j+m} <t \leq 2^{-j+m+1}}  \big| \mathbf{K}_t\left( 
        a  \right)\big|  \Bigg\|_{L_k^1(\mathbb{R}^n)}.
     \end{align*} Employing the semigroup property, along with \eqref{heat kernel esti_2}, and \eqref{kernel estimate_1}, for $2^{-j+m}< t \leq 2^{-j+m+1}$, we have 
     \begin{align*}
         \int_{\mathbb{R}^n}& k_t(x,y)|a(y)|w_k(y)dy\\
         & =    \int_{\mathbb{R}^n} \int_{\mathbb{R}^n} k_{t-2^{-j+m+1}}(x,z)k_{2^{-j+m+1}}(z,y)\,w_k(z)dz\,|a(y)|w_k(y)dy\\
         & \leq C \int_{\mathbb{R}^n} \int_{\mathbb{R}^n}  \frac{exp\left( -\frac{d_\mathcal{O}(x,z)^2}{2^{-j+m+1}}\right)}{\omega_k(B(z,2^{-j+m}))} k_{2^{-j+m+1}}(z,y)\,w_k(z)dz\,|a(y)|w_k(y)dy.
     \end{align*}
     Taking integration on both sides with respect to $w_k(x)\,dx$, and applying \eqref{k-kernel-estimate} together with Tonelli's theorem, we obtain
\begin{align*}
         \Bigg\|  & \sup_{ 2^{-j+m} <t \leq 2^{-j+m+1}}  \big| \mathbf{K}_t\left( 
        a  \right)\big|  \Bigg\|_{L_k^1(\mathbb{R}^n)} \\
        & \leq C \int_{\mathbb{R}^n} |a(y)| \int_{\mathbb{R}^n}  k_{2^{-j+m+1}}(z,y)\,w_k(z)dz\,w_k(y)dy\\
        & \leq C \int_{\mathbb{R}^n} |a(y)| \int_{\mathbb{R}^n} \frac{exp\left( \frac{d_\mathcal{O}(z,y)} {c2^{-j+m+1}}\right)}{\omega_k\left(B(z,2^{(-j+m+1)/2})\right)} \left(1+ \frac{2^{(-j+m+1)/2}}{\rho_k(y)} \right)^{-N} w_k(z)dz\,w_k(y)dy.
     \end{align*}Since $\operatorname{supp}(a) \subset B(x_0, r_0)$, we have $\rho_k(y) \sim \rho_k(x_0) = 2^{-j/2}$. Therefore,
\begin{align*}
         \sum_{m=0}^\infty \Bigg\|  \sup_{ 2^{-j+m} <t \leq 2^{-j+m+1}}
         \big| \mathbf{K}_t(a) \big|
         \Bigg\|_{L_k^1(\mathbb{R}^n)}
         &\leq C \sum_{m=0}^\infty 2^{-mN/2}
         \|a\|_{L_k^1(\mathbb{R}^n)} \leq C.
\end{align*}
Hence,
\[
\|f\|_{H^1_{\Tilde{\mathcal{L}}_k}}
=
\Bigg\|
\sum_{l=1}^\infty c_l a_l
\Bigg\|_{H^1_{\Tilde{\mathcal{L}}_k}}
\leq
\sum_{l} |c_l| \,
\|a_l\|_{H^1_{\Tilde{\mathcal{L}}_k}}
\leq
C\,
\|f\|_{H_{\Tilde{\mathcal{L}}_k}^{1,\mathrm{atm}}}.
\]
\end{proof}

\subsection{Characterization of dual space}

The following results concern the duality between $H_{\Tilde{\mathcal{L}}_k}^1$ and $BMO(\mathcal{L}_k)$. 
Our approach follows Stein’s method, which is based on the atomic decomposition of Hardy spaces \cite{Stein-book-1}. 
We recall below the definition of a $(1,q)_{\rho_k}$-atom.
\begin{definition}[cf.\cite{Stein-book-1}]
Let $1<q\leq \infty$ and let $a$ be a measurable function on $\mathbb{R}^n$. 
We say that $a$ is a $(1,q)_{\rho_k}$-atom if the following conditions hold:
\begin{enumerate}[(i)]
\item $\operatorname{supp}(a)\subset B(x_0,r_0)$;
\item $\|a\|_{L_k^q(\mathbb{R}^n)} \le \omega_k(B(x_0,r_0))^{-1/p}$, 
where $p$ is the conjugate exponent of $q$;
\item if $r_0 < \rho_k(x_0)$, then
\[
\int_{\mathbb{R}^n} a(x)\,w_k(x)\,dx = 0.
\]
\end{enumerate}
\end{definition}

\begin{theorem}
    There exists a linear isomorphism between the Banach spaces $\left( H_{\Tilde{\mathcal{L}}_k}^1\right)^*$ and $BMO(\mathcal{L}_k)$.
\end{theorem}

\begin{proof}
Let $f \in BMO(\mathcal{L}_k)$, we define $\Phi_f$ as 
\begin{align*}
    \Phi_f(h) = \int_{\mathbb{R}^n} f(x)h(x)w_k(x)dx \quad \text{for } h\in H_{\Tilde{\mathcal{L}}_k}^1.
\end{align*}
By the atomic decomposition of $H_{\Tilde{\mathcal{L}}_k}^1$, it suffices to check the boundedness of  $\Phi_f$ on $H_{\Tilde{\mathcal{L}}_k}^{1} \text{-atoms}$. Let $a$ be an atom with $\operatorname{supp}(a)\subseteq B_0=B(x_0,r_0)$. \\

\noindent  If $r_0 < \rho_k(x_0)$, then using the cancellation property of $a$ we obtain
\begin{align*}
|\Phi_f(a)|
&= \left| \int_{\mathbb{R}^n} f(x)a(x)w_k(x)\,dx \right|
 = \left| \int_{B_0}\big(f(x)-f_{B_0}\big)a(x)w_k(x)\,dx \right| \\
&\le \frac{1}{\omega_k(B_0)} \int_{B_0} |f(x)-f_{B_0}|\, w_k(x)\,dx
\le \|f\|_{BMO(\mathcal{L}_k)} .
\end{align*}
On the other hand, if $r_0 \geq \rho_k(x_0)$, we use the size condition of $a$ and obtain
\begin{align*}
|\Phi_f(a)|
= \left| \int_{\mathbb{R}^n} f(x)a(x)w_k(x)\,dx \right|
\le \frac{1}{\omega_k(B_0)} \int_{B_0} |f(x)|\, w_k(x)\,dx
\le \|f\|_{BMO(\mathcal{L}_k)} .
\end{align*}
Since finite linear combinations of atoms are dense in $H_{\Tilde{\mathcal{L}}_k}^1$, it follows that $\Phi_f$ extends to a continuous linear functional on $H_{\Tilde{\mathcal{L}}_k}^1$, with
\[
\|\Phi_f\| \leq \|f\|_{BMO(\mathcal{L}_k)}.
\]
We now consider $\Phi \in \left(H_{\Tilde{\mathcal{L}}_k}^1\right)^*$. Let $B_0 = B(x_0,r_0)$ with $r_0 \ge \rho_k(x_0)$, and let $g \in L_k^2(B_0)$ be nonzero. Define
\[
a = \frac{g}{\|g\|_{L_k^2(B_0)}\, \omega_k(B_0)^{1/2}} .
\]
Then $a$ is a $(1,2)$-atom. Since $\Phi$ is continuous, we have
\[
|\Phi(a)| \le \|a\|_{H_{\Tilde{\mathcal{L}}_k}^1}\,\|\Phi\|.
\]
Hence, for every $g \in L_k^2(B_0)$,
\[
|\Phi(g)| \le C \|\Phi\| \|g\|_{L_k^2(B_0)} \omega_k(B_0)^{1/2},
\]
where $C$ is the constant appearing in Theorem~\ref{atomic decomposition}. It follows that $\Phi$ defines a bounded linear functional on $L_k^2(B_0)$ with operator norm bounded by $C\|\Phi\|\omega_k(B_0)^{1/2}$. By the Riesz representation theorem, there exists $f_\Phi^{B_0} \in L_k^2(B_0)$ such that
\[
\Phi(g) = \int_{B_0} g(x) f_\Phi^{B_0}(x) w_k(x)\,dx,
\]
and
\[
\|f_\Phi^{B_0}\|_{L_k^2(B_0)} \le C \|\Phi\| \omega_k(B_0)^{1/2}.
\]
\noindent If $B(y,r_1) \subset B(y,r_2)$ with $r_1 \ge \rho_k(y)$, then for every $g \in L_k^2(B(y,r_1))$,
\[
\Phi(g)
= \int_{B(y,r_1)} g(x) f_\Phi^{B(y,r_1)}(x) w_k(x)\,dx
= \int_{B(y,r_1)} g(x) f_\Phi^{B(y,r_2)}(x) w_k(x)\,dx.
\]
Hence $f_\Phi^{B(y,r_1)} = f_\Phi^{B(y,r_2)}$ a.e. on $B(y,r_1)$. Consequently, there exists a unique function $f \in L_{k,\mathrm{loc}}^2(\mathbb{R}^n)$ such that
\[
\Phi(g) = \int_B g(x) f(x) w_k(x)\,dx
\]
for every $g \in L_k^2(\mathbb{R}^n)$ with $\operatorname{supp}(g)\subset B = B(y,r)$ and $r \ge \rho_k(y)$. In particular, this holds for all $(1,\infty)$-atoms and their finite linear combinations. Hence, $\Phi$ admits an integral representation on a dense subset and extends to the whole space $H_{\Tilde{\mathcal{L}}_k}^1$. It remains to show that $f \in BMO(\mathcal{L}_k)$. Let $B = B(y,r)$ with $r \geq \rho_k(y)$. Then
\[
\frac{1}{\omega_k(B)} \int_B |f(x)| w_k(x)\,dx
\le \frac{\|f\|_{L_k^2(B)}}{\omega_k(B)^{1/2}}
\le C \|\Phi\|.
\]
For balls $B = B(y,s)$ with $s < \rho_k(y)$, take a nonzero $g \in L_k^2(B)$ satisfying
\[
\int_B g(x) w_k(x)\,dx = 0.
\]
In this case, the representative $f$ is determined by taking the modulo with constants, that is, as an element of $L_k^2(B)/\mathcal C$. Moreover,
\[
\inf_{c \in \mathcal C} \|f-c\|_{L_k^2(B)}
\le C \|\Phi\| \omega_k(B)^{1/2}.
\]
By Corollary~\ref{Corollary J-N},
\[
\frac{1}{\omega_k(B)} \int_B |f-f_B| w_k(x)\,dx
\le \left(
\frac{1}{\omega_k(B)} \int_B |f-f_B|^2 w_k(x)\,dx
\right)^{1/2}
\le C \|\Phi\|,
\]
for every ball whose radius is smaller than the critical radius. Therefore $f \in BMO(\mathcal{L}_k)$.
\end{proof}

\section{Maximal operators on $BMO(\mathcal{L}_k)$ space}\label{S:5}
In this section, we study the behavior of maximal operators on the space $BMO(\mathcal{L}_k)$. In the classical setting, the maximal operator on $BMO$ is either identically infinite or satisfies
\[
\|\mathcal{M}f\|_{BMO} \leq C \|f\|_{BMO}.
\]
Bennett \textit{et al.} proved the boundedness of the maximal function restricted to a cube $Q$ on the space $BMO(Q)$ \cite{Bennett}. Later, Dziuba\'nski \textit{et al.} established the boundedness of maximal operators on the space $BMO(\mathcal{L})$. Similar results also hold in the setting of the Heisenberg group \cite{Lin} and metric measure spaces \cite{Dafni-2012}.
 In our analysis, we use the weighted uncentered Hardy-Littlewood maximal function,  which we denote as $\mathcal{M}_k$ and defined 
 \begin{align*}
     \mathcal{M}_kf(x) =  \sup\limits_{x\in B} \frac{1}{\omega_k(B)} \int_{B}|f(y)|w_k(y)dy  \quad \text{ for } f\in L_{k, \text{loc}}^1 .
 \end{align*}
In the Dunkl setting, there are many versions of the maximal operators. For more details, readers can refer to \cite{Almeida}.
 We decompose the family of all balls in $\mathbb{R}^n$ into two disjoint subcollections according to their radii. Define
     \begin{align*}
         \mathcal{B}_1 & = \{B= B(x, r): r \leq {\rho_k}(x) \text{ where } x\in \mathbb{R}^n \} 
         \end{align*} and
         \begin{align*}
          \mathcal{B}_2 & = \{B= B(x, s): s > {\rho_k}(x) \text{ where } x\in \mathbb{R}^n \}. 
     \end{align*}
Associated with the subcollection $\mathcal{B}_1$ and $\mathcal{B}_2$, we define the localized maximal operator 
\begin{align*}
    \mathcal{M}_{\mathcal{B}_1}f(x) = \sup_{\substack{B \in \mathcal{B}_1 \\ x \in B}} \frac{1}{\omega_k(B)} \int_B |f(y)|\, w_k(y)dy, 
\end{align*}
and 
\begin{align*}
  \mathcal{M}_{\mathcal{B}_2}f(x) = \sup_{\substack{B \in \mathcal{B}_2 \\ x \in B}} \frac{1}{\omega_k(B)} \int_B |f(y)|\, w_k(y)dy, 
\end{align*}
respectively. Then 
\begin{align}\label{maximal functiona}
    \mathcal{M}_k f(x) = \max\{\mathcal{M}_{\mathcal{B}_1}f(x), \mathcal{M}_{\mathcal{B}_2}f(x)\}.
\end{align}

We begin by defining the maximal function over the $BMO(\mathcal{L}_k)$ space and further discuss the boundedness property.   Before proving the boundedness of the weighted maximal function over the $BMO(\mathcal{L}_k)$ space, we observe some properties of $\mathcal{M}_k$ function over $BMO(\mathcal{L}_k)$.

\begin{proposition}
Let $x_0\in \mathbb{R}^n$, $C_0 \geq 1$, and $f\in BMO(\mathcal{L}_k)$. Then $\mathcal{M}_k f(x)$ is finite almost everywhere on $B(x_0, C_0{\rho_k}(x_0))$.
\end{proposition}

\begin{proof}
Decompose $f = f_1 + f_2$, where $f_1 = f \chi_{B(x_0, C_0 {\rho_k}(x_0))^*}$ and $f_2 = f \chi_{\mathbb{R}^n \setminus B(x_0, C_0 {\rho_k}(x_0))^*}$. Since $f_1 \in L_k^1(\mathbb{R}^n)$, the weighted maximal operator $\mathcal{M}_k f_1$ is finite almost everywhere on $\mathbb{R}^n$.

\noindent Now we consider $\mathcal{M}_k f_2$. Let $x \in B(x_0, C_0 {\rho_k}(x_0))$. For any ball $B $ containing $x$,
\begin{align*}
    \frac{1}{\omega_k(B)} \int_B |f_2(\xi)|\, w_k(\xi)\, d\xi \neq 0
\quad \text{only if} \quad
B \cap B(x_0, C_0 {\rho_k}(x_0))^{*} \neq \emptyset.
\end{align*}
Writing $B = B(y,r_y)$, this implies $2r_y > C_0 {\rho_k}(x_0)$, and hence $B \subset B(x_0, 4 r_y)$ and 
\begin{align*}
    \frac{1}{\omega_k(B(y, r_y))} \int_{B(y, r_y)} |f_2(\xi)| w_k(\xi) \, d\xi \leq \frac{1}{\omega_k(B(y, r_y))} \int_{B(x_0, 4 r_y)} |f(\xi)| w_k(\xi) \, d\xi.
\end{align*}
Since $f_1 \in L_k^1(\mathbb{R}^n)$, it follows from the weak-type $(1,1)$ boundedness of $\mathcal{M}_k$ that $\mathcal{M}_k f_1(x)$ is finite almost everywhere on $\mathbb{R}^n$. Consequently, we have
\begin{align}
    \frac{1}{\omega_k(B(y, r_y))} & \int_{B(y, r_y)} |f_2(\xi)| w_k(\xi) \, d\xi \notag \\
    & \leq C_{w_k} \frac{\omega_k(B(x_0, 4 r_y))}{\omega_k(B(y, r_y))}  \frac{1}{\omega_k(B(x_0, 4 r_y))} \int_{B(x_0, 4 r_y)} |f(\xi)| w_k(\xi) \, d\xi \notag\\
    &\leq C_{w_k}^2 4^{n + 2 \gamma} \|f\|_{BMO(\mathcal{L}_k)} < \infty. \label{Eq_5.1}
\end{align}
\end{proof}
In the following proposition, we establish the local integrability of the function $\mathcal{M}_kf$, whenever $f\in BMO(\mathcal{L}_k)$
\begin{proposition}
Let $f\in BMO(\mathcal{L}_k)$. Then $\mathcal{M}_kf \in L^1_{k,\text{loc}}(\mathbb{R}^n)$. 
\end{proposition}
\begin{proof}
To prove $\mathcal{M}_kf \in L^1_{k,\text{loc}}(\mathbb{R}^n)$, it follows from \eqref{maximal functiona} that it is enough to show that
\(
\mathcal{M}_{\mathcal{B}_1}f,\,
\mathcal{M}_{\mathcal{B}_2}f \in L_k^1(B_0),
\)
where $B_0 = B(x_0,r_0)$ with $r_0 < \infty$ and $x_0 \in \mathbb{R}^n$.\\

\noindent For $\mathcal{M}_{\mathcal{B}_2}f$, it follows directly from the fact that $f \in BMO(\mathcal{L}_k)$ that
\[
  \int_{B_0} \mathcal{M}_{\mathcal{B}_2}f(x)\,w_k(x)\,dx
  \leq
  \omega_k(B_0)\|f\|_{BMO(\mathcal{L}_k)}
  < \infty.
\]
On the other hand, for every $x\in B_0$,  \cite[Lemma 8 ]{Hejna-2021} gives the existence of $c, \kappa_0>0$ so that
\begin{align} \label{ineq*}
    c^{-1} {\rho_k}(x_0) \left(1+ \frac{|x-x_0|}{{\rho_k}(x_0)} \right)^{-\kappa_0} \leq {\rho_k}(x) \leq c {\rho_k}(x_0) \left(1+ \frac{|x-x_0|}{{\rho_k}(x_0)} \right)^{\frac{\kappa_0}{\kappa_0+1}} .
\end{align} 
The inequality \eqref{ineq*} rewritten as 
\begin{align} \label{Ineq*-1}
    c^{-1}  \left(1+ \frac{r_0}{{\rho_k}(x_0)} \right)^{-\kappa_0} \leq \frac{{\rho_k}(x)}{{\rho_k}(x_0)} \leq c \left(1+ \frac{r_0}{{\rho_k}(x_0)} \right)^{\frac{\kappa_0}{\kappa_0+1}}.
\end{align}
Moreover, if $x\in B_0 \cap B $ and $B\in \mathcal{B}_1$, where $B=B(y, r)$, then 
\begin{align}\label{Ineq*-2}
    c^{-1} {\rho_k}(y) 2^{-\kappa_0} \leq {\rho_k}(x) \leq c {\rho_k}(y)2^{\frac{\kappa_0}{\kappa_0+1}}.
\end{align}
Combining \eqref{Ineq*-1} and \eqref{Ineq*-2}, we have 
\begin{align*}
    (c^{-1})^2 2^{-\frac{\kappa_0}{\kappa_0+1}}\left(1+ \frac{r_0}{{\rho_k}(x_0)} \right)^{-\kappa_0}  \leq  \frac{{\rho_k}(y)}{{\rho_k}(x_0)} \leq  c^2 2^{\left(\kappa_0+\frac{\kappa_0}{\kappa_0+1}\right)} \left(1+ \frac{r_0}{{\rho_k}(x_0)} \right)^{\frac{\kappa_0}{\kappa_0+1}}. 
\end{align*}
Let $R_0=r_0+(1+2C){\rho_k}(x_0)$, where $C=c^2 2^{\left(\kappa_0+\frac{\kappa_0}{\kappa_0+1}\right)} \left(1+ \frac{r_0}{{\rho_k}(x_0)} \right)^{\frac{\kappa_0}{\kappa_0+1}}$. Then $B\subset \widetilde{B_0}=B(x_0, R_0)$ for all balls $B$ such that $B\in \mathcal{B}_1$ and $B\cap B_0 \neq \emptyset$.
Therefore, for $x\in B_0$ we have 
\begin{align*}
    \mathcal{M}_{\mathcal{B}_1}f(x) \leq \mathcal{M}_k(f\chi_{\widetilde{B_0}})(x).
\end{align*} At this point, Corollary \ref{Corollary J-N} implies that $f\chi_{\widetilde{B_0}} \in L_k^2(\mathbb{R}^n)$. Invoking the $L^2$ boundedness of the weighted maximal function \cite{Alto-2011} and H\"older's inequality reduces that 
\begin{align*}
    \int_{B_0}\mathcal{M}_k(f\chi_{\widetilde{B_0}})(x)w_k(x)dx &\leq \omega_k(B_0)^{\frac{1}{2}}\|\mathcal{M}_k(f\chi_{\widetilde{B_0}})\|_{L_k^2(\mathbb{R}^n)}\\
    & \leq C  \omega_k(B_0)^{\frac{1}{2}} \|f\|_{L_k^2({\widetilde{B_0}})}\\
    & \leq C w_k(\widetilde{B_0})\|f\|_{BMO(\mathcal{L}_k)} < \infty.
\end{align*} Thus, we have the integrability.
\end{proof}
\begin{note}
Observe that we have shown the local integrability of the function $\mathcal{M}_kf$ whenever $f\in BMO(\mathcal{L}_k)$. In the same way one  can  prove the local $p$-integrability of $\mathcal{M}_kf$, where $p\in [1, \infty)$. Using the fact $f\in BMO(\mathcal{L}_k)$, we have  $$\int_{B_0}|\mathcal{M}_{\mathcal{B}_2}f(x)|^p w_k(x)dx \leq \omega_k(B_0)\|f\|^p_{BMO(\mathcal{L}_k)}.$$ 
Together with the $L^p$ boundedness of the weighted maximal function \cite{Alto-2011} and Corollary \ref{Corollary J-N}, it is straightforward that
\begin{align*}
\int_{B_0}|\mathcal{M}_{\mathcal{B}_1}f(x)|^pw_k(x)dx & \leq \int_{B_0}|\mathcal{M}_k(f\chi_{\widetilde{B_0}})(x)|^pw_k(x)dx\\
    &\leq \|\mathcal{M}_k(f\chi_{\widetilde{B_0}})\|^p_{L_k^p(\mathbb{R}^n)}\\
    & \leq C\|f\|^p_{L_k^p({\widetilde{B_0}})}\\
    & \leq C w_k(\widetilde{B_0})\|f\|^p_{BMO(\mathcal{L}_k)} < \infty.
\end{align*}
\end{note}
In the following theorem, we prove the boundedness of the maximal function on the space $BMO(\mathcal{L}_k)$. Our approach is inspired by the method developed in  \cite{Bennett} and \cite{Lin}. 
\begin{theorem}
    The maximal operator $\mathcal{M}_k$ is a bounded operator on $BMO(\mathcal{L}_k)$ space.
\end{theorem}
\begin{proof} We split the proof into two cases based on the two cases in the definition of $BMO(\mathcal{L}_k)$.
 \begin{enumerate}[$i)$]

 \item Case-1: $B=B(x_0,r_0)$ for $x_0\in \mathbb{R}^n$ and $r_0\geq {\rho_k}(x_0)$. Here we have to show 
     \begin{align} \label{case-2}
         \frac{1}{\omega_k(B)}\int_B |\mathcal{M}_kf(y)|w_k(y)dy \,\leq \,C\|f\|_{BMO(\mathcal{L}_k)}.
     \end{align} To this end, we decompose the function $f=f_1+f_2$, where $f_1= f\chi_{B^*}$ and $f_2=f\chi_{\mathbb{R}^n\setminus B^*}$.  In light of  \eqref{Eq_5.1}, we observe that 
     \begin{align} \label{f_2}
         \frac{1}{\omega_k(B)}\int_B |\mathcal{M}_kf_2(y)|w_k(y)dy \,\leq \,C\|f\|_{BMO(\mathcal{L}_k)}.
     \end{align} 
     For the function $f_1$, we use the boundedness of weighted maximal function $\mathcal{M}_k$  on $L_k^2(\mathbb{R}^n)$ \cite{Alto-2011} and Corollary \ref{Corollary J-N} yields that 
     \begin{align}
         \frac{1}{\omega_k(B)}\int_B |\mathcal{M}_kf_1(x)|w_k(x)dx& \leq \frac{1}{\omega_k(B)}\left( \int_B |\mathcal{M}_kf_1(x)|^2w_k(x)dx \right)^{\frac{1}{2}} (\omega_k(B))^{\frac{1}{2}}\notag \\
         &  \leq C \left(\frac{1}{\omega_k(B)} \int_B |f_1(x)|^2w_k(x)dx \right)^{\frac{1}{2}} \notag \\
          &  \leq C \left(\frac{1}{\omega_k(B)} \int_B |f(x)|^2w_k(x)dx \right)^{\frac{1}{2}} \notag \\
          & \leq C \|f\|_{BMO(\mathcal{L}_k)}. \label{f_1}
     \end{align}

     \item Case-2: $B_0=B(x_0, r_0)$ for $x_0 \in \mathbb{R}^n$ and $r_0 < {\rho_k}(x_0)$. The desired estimate we have to prove is 
     \begin{align} \label{case-1}
         \frac{1}{\omega_k(B_0)} \int_{B_0} |\mathcal{M}_kf(y)- (\mathcal{M}_kf)_{B_0}\big| w_k(y)dy \leq C\|f\|_{BMO(\mathcal{L}_k)}.
     \end{align} We may assume that $f$ is nonnegative function.
     We decompose the family of all balls in $\mathbb{R}^n$ into two disjoint subcollections as 
     \begin{align*}
         \mathcal{B}^1 & = \{B:  B= B(y_0, r) \subset B_0^*, \text{ where $y\in \mathbb{R}^n$ and $r>0$}\} 
         \end{align*} and
         \begin{align*}
          \mathcal{B}^2 & = \{B: B= B(y_0, r): B \not\subset B_0^*, \text{ where $y\in \mathbb{R}^n$ and $r>0$} \}. 
     \end{align*}
Associated with the subcollection $\mathcal{B}^1$ and $\mathcal{B}^2$, we define the localized maximal operator 
\begin{align*}
   \mathcal{M}_{\mathcal{B}^1}f(x) =  \sup_{\substack{B \in \mathcal{B}^1 \\ x \in B}} \frac{1}{\omega_k(B)} \int_B f(y)\, w_k(y)dy, 
\end{align*}
and 
\begin{align*}
  \mathcal{M}_{\mathcal{B}^2}f(x)  = \sup_{\substack{B \in \mathcal{B}^2 \\ x \in B}} \frac{1}{\omega_k(B)} \int_B f(y)\, w_k(y)dy, 
\end{align*}
respectively. Clearly $\mathcal{M}_k f(x) = \max\{\mathcal{M}_{\mathcal{B}^1}f(x), \mathcal{M}_{\mathcal{B}^2}f(x)\}$, and if $ \Omega_0=\{x\in B_0:\, \mathcal{M}_kf(x) \geq (\mathcal{M}_kf)_{B_0} \}, \,  \Omega_1= \{x\in \Omega_0: \mathcal{M}_{\mathcal{B}^1}f(x) \geq \mathcal{M}_{\mathcal{B}^2}f(x)\}, \text{ and } \Omega_2 = \Omega_0\setminus\Omega_1. $ Then 
    \begin{align*}
        \frac{1}{\omega_k(B_0)} &\int_{B_0} |\mathcal{M}_kf(x) - \left( \mathcal{M}_kf\right)_{B_0}|w_k(x)dx \\  &= \frac{2}{\omega_k(B_0)} \int_{\Omega_0} \left( \mathcal{M}_kf(x) - \left( \mathcal{M}_kf\right)_{B_0} \right) w_k(x)dx\\
        & = \frac{2}{\omega_k(B_0)}
        \sum\limits_{j=1}^2\int_{\Omega_j} \left( \mathcal{M}_{\mathcal{B}^j}f(x) - \left( \mathcal{M}_kf\right)_{B_0} \right) w_k(x)dx. 
    \end{align*}
We deduce \eqref{case-1} by establishing that 
\begin{align*} 
   \frac{1}{\omega_k(B_0)} \int_{\Omega_j} \left( \mathcal{M}_{\mathcal{B}^j}f(x) - \left( \mathcal{M}_kf\right)_{B_0} \right) w_k(x)dx \leq C \|f\|_{BMO(\mathcal{L}_k)} \quad \text{ for } j=1,2.
\end{align*}
Consider the function $\Tilde{f_1}= \left( f-f_{B_0^*}\right)\chi_{B_0^*}$.  Since  $f_{B_0^*} \leq \left(\mathcal{M}_{k}f\right)_{B_0} $, we have $$\mathcal{M}_{\mathcal{B}^1}f(x) \leq \mathcal{M}_k(\Tilde{f_1})(x)+\left(\mathcal{M}_{k}f\right)_{B_0}.$$ Then Corollary \ref{Corollary J-N} implies that
\begin{align*}
    \frac{1}{\omega_k(B_0)} \int_{\Omega_1} \left( \mathcal{M}_{\mathcal{B}^1}f(x) - \left( \mathcal{M}_kf\right)_{B_0} \right) w_k(x)dx 
    & \leq  \frac{1}{\omega_k(B_0)} \int_{\Omega_1} {\mathcal{M}_k}\Tilde{f_1}(x) w_k(x)dx \\
    & \leq \left( \frac{1}{\omega_k(B_0)} \int_{\Omega_1} {\mathcal{M}_k}\Tilde{f_1}(x)^2 w_k(x)dx \ \right)^{\frac{1}{2}}\\
    & \leq \left( \frac{C}{\omega_k(B_0)} \int_{B_0} \Tilde{f_1}(x)^2 w_k(x)dx \ \right)^{\frac{1}{2}} \\
    & \leq C^* \|f\|_{BMO(\mathcal{L}_k)}.
\end{align*} 
On the other hand, let $x\in \Omega_2$ and $B^\prime \in \mathcal{B}^2$ with $B^\prime \cap (B_0^*)^c \neq \emptyset$. Consider $B^{\prime\prime}=B^{\prime ***}$. Then  $B_0\subset B^{\prime\prime}$, we have $f_{B^{\prime\prime}} \leq \left(\mathcal{M}_kf\right)_{B_0}$. Thus we obtained 
\begin{align*}
    f_{B^\prime}-\left(\mathcal{M}_kf\right)_{B_0} \leq f_{B^\prime} - f_{B^{\prime\prime}} \leq C\|f\|_{BMO(\mathcal{L}_k)}.
\end{align*}
Taking supremum over $\mathcal{B}^2$, we get
\begin{align*}
    \mathcal{M}_{\mathcal{B}^2}f(x)- \left(\mathcal{M}_kf\right)_{B_0} \leq C\|f\|_{BMO(\mathcal{L}_k)}.
\end{align*} Consequently,
\begin{align*}
    \frac{1}{\omega_k(B_0)}\int_{\Omega_2} \left( \mathcal{M}_{\mathcal{B}^2}f(x) - \left( \mathcal{M}_kf\right)_{B_0} \right) w_k(x)dx \leq C \|f\|_{BMO(\mathcal{L}_k)}.
\end{align*}This completes the proof.
 \end{enumerate}
\end{proof}
\noindent\textbf{Acknowledgment:}
      The authors extend their gratitude for the financial support provided by the UGC Government of India (221610147795) for the first author and Anusandhan National Research Foundation (ANRF) with the reference number-SUR/2022/005678, for the second author. Their funding and resources were instrumental in the completion of this work.   \\
\textbf{Declarations:}
\\
\textbf{Conflict of interest:}
We would like to declare that we do not have any conflict of interest.
\\
\textbf{Data availability:} No data source is needed.
 

\begin{thebibliography}{Md.}
 
\bibitem{Alto-2011} D. Aalto, L. Berkovits, O.E. Kansanen, H. Yue, John-{N}irenberg lemmas for a doubling measure. Studia Math. \textbf{204}(1):21-37 (2011). 

\bibitem{Almeida} V.Almeida, J.J. Betancor, J.C. Fariña, L. Rodr\'iguez-Mesa, Maximal, Littlewood-Paley, Variation, and oscillation operators in the Dunkl setting. J. Fourier Anal. Appl. \textbf{30}(5):60(1-41) (2024).
 
\bibitem{Amri-2018}B. Amri, A. Hammi,  Dunkl-Schr\"odinger operators. Complex Anal. Oper. Theory. \textbf{13}(3):1033-1058 (2019).

\bibitem{Amri-2021}B. Amri, A. Hammi, Semigroup and Reisz transform for the Dunkl-Schr\"odinger operators. Semigroup Forum. \textbf{101}(3):507-533 (2020).

\bibitem{Hejna-2019}J.P. Anker, J. Dziuba\'nski, A.  Hejna,  Harmonic functions, conjugate harmonic functions and the {H}ardy space {$H^1$} in the rational {D}unkl setting. J. Fourier Anal. Appl. \textbf{25}(5):2356-2418 (2019).

\bibitem{Badr}N. Badr, B. Ben Ali, $L^p$ Boundedness of the Riesz transform related to Schr\"odinger operators on a manifold. Ann. Scuola Norm. Sup. di Pisa Cl. Sci. 5. \textbf{8}(4):725-765 (2009).

\bibitem{Bennett}C. Bennett, R.A. DeVore, R. Sharpley, 
Weak-{$L\sp{\infty }$}\ and {BMO}. Ann. of Math. \textbf{113}(3):{601-611} (1981).



\bibitem{Cherednik} I. Cherednik, Double Affine Hecke Algebras, London Math. Soc. Lect. Note Ser., vol. 319, Cambridge Univ. Press, Cambridge (2005).

\bibitem{Christ}M. Christ, D. Geller, Singular integral characterizations of Hardy spaces on homogeneous groups. Duke Math. J. \textbf{51}:547-598 (1984).


\bibitem{Dafni-2012}G. Dafni, H. Yue, Some characterizations of local bmo and $h^1$ on metric measure spaces. Anal. Math. Phys. \textbf{2}(3):285-318 (2012).

\bibitem{DJ1} M.F.E. de Jeu, The Dunkl transform. Invent. Math. \textbf{113}(1):147-162 (1993).

 \bibitem{D2}C.F. Dunkl, Differential-difference operators associated to reflection groups. Trans. Amer. Math. Soc. \textbf{311}(1):167-183 (1989). 

 \bibitem{Dunkl-Xu}C.F. Dunkl, Y. Xu, Orthogonal Polynomials of Several Variables, Encyclopedia Math. Appl., vol. 81, Cambridge Univ. Press, (2001).


\bibitem{Dziubanski-1999} J. Dziuba\'nski, J. Zienkiewicz, Hardy spaces {$H^1$} associated to Schrodinger operator with potential satisfying reverse H\"older inequality. Rev. Mat. Iberoam. \textbf{15}(2):277-95 (1999).

\bibitem{Dziubanski-2005}J. Dziuba\'nski, G.  Garrig\'os,  T. Mart\'inez, J.L. Torrea,  J. Zienkiewicz,  $BMO$ spaces related to Schr\"odinger operators with potentials satisfying a reverse H\"older inequality. Math. Z.  \textbf{249}(2):329-356 (2005).


\bibitem{Jacek-2004} J. Dziuba\'nski, J. Zienkiewicz, Hardy spaces {$H^1$} for {S}chr\"odinger operators with
certain potentials. Studia Math. \textbf{164}(1):39-53 (2004).

\bibitem{Jacek-2023}J. Dziuba\'nski, A. Hejna, A note on commutators of singular integrals with {BMO} and {VMO} functions in the {D}unkl setting. Math. Nachr. \textbf{297}(2):629-643 (2023).


\bibitem{Jacek-2025} J. Dziuba\'nski, A. Hejna,
On {D}unkl {S}chr\"odinger semigroups with {G}reen bounded potentials. Constr. Approx. \textbf{61}(3):481-509 (2025). 


\bibitem{Etingof} P. Etingof, Calogero-Moser systems and representation theory, Zurich Lectures in Advanced
Mathematics, European Mathematical Society (EMS), Z\"urichl (2007).

\bibitem{Fefferman-1972}C. Fefferman, E.M. Stein, $H^p$ spaces of several variables. Acta Math. \textbf{129}(3–4):137–193 (1972).


\bibitem{Goldberg} D. Goldberg, A local version of real {H}ardy spaces. Duke Math. J. \textbf{46}(1):27-42 (1979).

\bibitem{Graczyk} P. Graczyk, M. R\"osler, M. Yor (Eds.), Harmonic and Stochastic Analysis of Dunkl Processes,Travaux en cours 71, Hermann, Paris (2008).

\bibitem{Guo}Q. Guo, J. Li, B.D. Wick, Fefferman–Stein type decomposition of CMO spaces in the Dunkl setting and an application. Nonlinear Anal. \textbf{262}:113916,27 (2026). 

 \bibitem{Heckman}G.J. Heckman,  Dunkl operators. Séminaire Bourbaki 828, 1996–97; Astérisque 245, 223-246 (1997).

\bibitem{Hejna-2020}A. Hejna, Hardy spaces for the {D}unkl harmonic oscillator. Math. Nachr. \textbf{293}(11):2112-2139 (2020).

\bibitem{Hejna-2021}A. Hejna, Schr\"odinger operators with reverse H\"older class potentials in the Dunkl setting and Hardy spaces. J. Fourier Anal. Appl. \textbf{27}(3):46,42 (2021).

\bibitem{Hejna-2021_2}A. Hejna, Behavior of eigenvalues of certain {S}chr\"odinger operators in the rational {D}unkl setting. Anal. Math. Phys. \textbf{11}(3):116 (2021).

\bibitem{Jiu}J. Jiu, Z. Li, The dual of the Hardy space associated with the Dunkl operators. Adv. Math. \textbf{412}:108810, 55 (2023).


\bibitem{John-Nirenberg-1961}F. John,  L. Nirenberg, On functions of bounded oscillation. Comm. Pure Appl. Math. \textbf{14}(3):415-426 (1961).

\bibitem{Lin_C}C.C. Lin, H. Liu, Y. Liu, Hardy spaces associated with Schr\"odinger operators on the Heisenberg group, preprint available at  arXiv:1106.4960. 

\bibitem{Lin} C.C. Lin, H. Liu, $BMO_L(\mathbb{H}^n) $ spaces and Carleson measures for Schr\"odinger operators, Adv. Math. \textbf{228}:1631-1688 (2011).

\bibitem{Macdonald}I.G. Macdonald, Affine Hecke Algebras and Orthogonal Polynomials, Cambridge Tracts Math. vol. 157, Cambridge Univ. Press, Cambridge (2003).

\bibitem{Opdam-2000}E.M. Opdam, Lecture notes on Dunkl operators for real and complex reflection Groups. MSJ Memoirs, vol.8,  Math. Soc. of Japan, Tokyo, (2000).

\bibitem{TV-15}M. R\"orsler,  A positive radial product formula for the Dunkl kernel. Trans. Amer. Math. Soc., \textbf{355}(6):2413-2438 (2003).

\bibitem{Rosler-03} M. R\"osler, Dunkl operators: theory and applications, in: E. Koelink, W. Van Assche (Eds.), Orthogonal Polynomials and Special Functions, in: Lecture Notes in Math. vol. 1817, Springer-Verlag, Berlin, 93-135 (2003).

\bibitem{Shen-1995} Z. Shen,  $L^p$ estimates for the Schr\"odinger operators with certain potentials.  Ann. Inst. Fourier \textbf{45}(2):513-546 (1995).

\bibitem{Stein-book-1}  E.M. Stein, Harmonic analysis: Real-variable methods, orthogonality, and oscillatory integrals, Princeton University Press, Princeton, NJ,
 (1993). 

\bibitem{Stein-book-2}E.M. Stein, Singular integrals and differentiability properties of functions. Princeton University Press, Princeton (1970).
 
\bibitem{Stein-1960} E.M. Stein, G. Weiss, On the theory of harmonic functions of several variables. I. The theory of $H^p$-spaces. Acta Math. \textbf{103}:25-62 (1960).

\end{thebibliography}
\end{document}